\documentclass{amsart}
\usepackage[utf8]{inputenc}
\usepackage[english]{babel}
\usepackage{amsmath,amsthm, amsfonts, amssymb, setspace, enumerate, graphicx, mathtools, booktabs, float, caption, subcaption, multicol, thmtools}
\usepackage[dvipsnames]{xcolor}

\theoremstyle{plain}
\newtheorem{thm}{Theorem}
\newtheorem*{thmintro1}{Theorem~\ref{thm: JCR1StoCR1S}}
\newtheorem*{thmintro2}{Theorem~\ref{thm: JCR2SasCR2S}}
\newtheorem*{thmintro3}{Theorem~\ref{thm: ContBridgeThm}}
\newtheorem{cor}{Corollary}
\newtheorem*{corintro}{Corollary~\ref{cor: DGThmForCRSD}}

\newtheorem{lem}{Lemma}

\theoremstyle{remark}
\newtheorem{rem}{Remark}

\newtheorem{exm}{Example}

\theoremstyle{definition}
\newtheorem{defn}{Definition}

\newcommand{\R}{\mathbb{R}}
\newcommand{\s}{\mathbb{S}}
\newcommand{\D}{\mathbb{D}}
\newcommand{\A}{\alpha}
\newcommand{\dd}{\partial}
\newcommand{\z}{\zeta}

\begin{document}
\title{On Contact Round Surgeries on $(\mathbb{S}^3, \xi_{st})$ and their Diagrams}

\author{Prerak Deep}
\address{Department of Mathematics, Indian Institute of Science Education and Research Bhopal}
\curraddr{}
\email{prerakd@iiserb.ac.in}
\thanks{}

\author{Dheeraj Kulkarni}
\address{Department of Mathematics,  Indian Institute of Science Education and Research Bhopal}
\curraddr{}
\email{dheeraj@iiserb.ac.in}
\thanks{}



\date{}

\dedicatory{}

\begin{abstract}
We introduce the notion of contact round surgery of index $1$ on Legendrian knots in a general contact 3-manifold. It generalizes the notion of contact round surgery of index 1 on Legendrian knots introduced by Adachi. In $\left(\mathbb{S}^3, \xi_{st}\right)$, we introduce the notion of contact round 
surgery of index 2 on a Legendrian knot and realize Adachi's contact round 2-surgery on a convex torus as a contact round surgery of index $2$ on a Legendrian knot in $\left(\s^3, \xi_{st}\right)$. 
We associate surgery diagrams to contact round surgeries of indices 1 and 2 on Legendrian knots in $\left(\mathbb{S}^3, \xi_{st}\right)$. With this set-up, we show that every closed connected contact 3-manifold can be obtained by 
performing a sequence of contact round surgeries on some Legendrian link in $\left(\mathbb{S}^3, \xi_{st}\right)$, thus obtaining a contact round surgery 
diagram for each contact 3-manifold. This is analogous to the result of Ding-Geiges for contact Dehn surgeries. We also discuss a bridge between certain 
pairs of contact round surgery diagrams of indices 1 and 2, and contact $(\pm1)$-surgery diagrams. We use this bridge to establish the result mentioned above. In 
the end, we derive a corollary that gives sufficient conditions on contact round surgeries to produce symplectically fillable manifolds.

\end{abstract}

\maketitle

\section{introduction}

Dehn surgery is a well known method in topology to obtain a new 3-manifold from a given 3-manifold.
In the 1960s, Lickorish and Wallace (cf.  \cite{L, W}) independently proved that any connected closed orientable 3-manifold can be obtained by an integral Dehn surgery on a link in $\s^3$. 
This result is now known as the \textit{Lickorish-Wallace theorem}. 
Moreover, an integral Dehn surgery along a link in $\s^3$ can be encoded in terms of a framed link, i.e., a link with integers associated with each of its components.
Thus, we get framed link presentations of a closed connected orientable 3-manifold in $\s^3$ as a corollary to Lickorish-Wallace theorem. 

In 1975, Asimov introduced round surgery on an $n$-manifold in his seminal work \cite{Asimov}. 
As an application to 3-manifolds, he proved that any compact connected 3-manifold can be obtained by a finite sequence of round surgeries on $\s^3$. 
However, unlike the Lickorish-Wallace theorem, Asimov's theorem does not provide a presentation of round surgeries in terms of framed link in $\s^3$. 
%In particular, we identified a gap in the literature observed that there is no combinatorial framework associated to round surgeries on 3-manifolds in terms of framed links.
The authors of this article addressed this gap in \cite{PD-RSD} by introducing round surgery diagrams similar to Dehn surgery diagrams. 
In particular, they showed that round surgeries on $\s^3$ can be encoded by framed links and further proved that every closed connected orientable 3-manifold can be obtained by round surgeries of indices 1 and 2 along a framed link in $\s^3$. 
Thus, we get round surgery diagrams representing 3-manifolds.  

In the world of contact 3-manifolds, a parallel set of results were proved.
In 1991, Weinstein (cf.~\cite{Weinstein}) introduced Legendrian surgery on a Legedrian link in a contact 3-manifold by observing the effect of symplectic handle addition to symplectic 4-manifolds on its boundary which is a contact 3-manifold. 
Later, in 1998, Gompf (cf.~\cite{Gompf-HBC_SteinSurface}) realized it as contact $(-1)$-surgery on a Legendrian knot. 
In 2001, Ding and Geiges (cf. \cite{DG-SymplecticFillings})  introduced rational contact surgery on Legendrian knots. 
Further, they proved that contact $(\pm1)$-surgeries were sufficient to obtain any closed connected contact 3-manifold from the standard contact 3-sphere. 
Their result is stated as follows. 
\begin{thm}[Ding-Geiges Theorem, \cite{DG}]\label{thm: LegPresnthm}
    Any closed connected (co-orientable) contact 3-manifold can be obtained by a contact $(\pm1)$-surgery on a Legendrian link in the standard contact 3-sphere.  
\end{thm}

The above result paves the way for presentation of contact 3-manifolds in terms of framed Legendrian links. Thus, we have contact surgery diagrams. It is important to note that a given contact 3-manifold admits infinitely many surgery diagrams similiar to the case of Dehn surgery diagrams for (topological) 3-manifolds. Kirby calculus on framed links provides a complete answer about when two Dehn surgery diagrams represent the same 3-manifold. Kirby calculus consists of two topological moves on framed links--known as Kirby move of Type 1 and Type 2.

In \cite{DG2}, Ding-Geiges provided a contact analogue of Kirby move of type 2 along with a few handle moves on a contact surgery diagram in a standard contact 3-sphere. 
Along the same lines, the authors of this article defined the contact analogue of Kirby move of type 1 in \cite{PD-CFKM}.

Round surgery has been studied for contact manifolds by Jiro Adachi in \cite{Jiro, Jiro17, Jiro19}.
In \cite{Jiro17}, Adachi introduced round 1-surgery on a transverse link and contact round 2-surgery on a convex torus. Using these contact round surgeries, he provided an alternate proof to Martinet's theorem and a contact version of Asimov's theorem.
%Moreover, he proved that any closed, connected contact 3-manifold can be obtained by a finite sequence of contact round surgeries on some transverse link and convex tori in the standard contact 3-manifold.  

In \cite{Jiro}, Adachi introduced symplectic round handles and their attachments. 
Like Weinstein's contact surgery, he defined contact round surgery on a contact $(2n+1)$-manifold $\left(M, \xi \right)$ as the change in the boundary $M\times \{1\}$ after attaching a symplectic round handle on the symplectization $M\times [0,1]$ of $M$. 
He proved that these contact round surgeries preserve the strong symplectic fillability of $\left(M, \xi \right)$. 
On a contact 3-manifold $M$, these contact round surgeries become equivalent to Legendrian contact round surgery, analogous to Weinstein Legendrian surgery.  
Moreover, in \cite{Jiro19}, Adachi defines contact round 1-surgery on a contact 3-manifold given by the above contact round surgery, i.e. using symplectic round handle attachment. 
In this article, we will refer to this surgery as \textit{Legendrian round surgery}.
In the same article, Adachi further defined another version of contact round 2-surgery along a convex torus in a contact 3-manifold $M$ as the change in the concave end $M\times \{0\}$ of the symplectization $M\times [0,1]$ after attaching the symplectic round 1-handle along their belt spheres. 
In this article, we are working with this version of contact round surgery of index 2. 
In Theorem \ref{thm: JCR2SasCR2S}, we realize this version of contact round surgery of index 2 along a \emph{convex torus} in $\left(\s^3, \xi_{st}\right)$ as a contact round surgery of index $2$ on a \emph{Legendrian knot} in the standard contact 3-sphere.
%In case of $\left(\s^3, \xi_{st}\right)$, we generalize Adachi's this version of contact round 2-surgery to a more general form of contact round 2-surgery on a Legendrian knot (See Theorem \ref{thm: ACR2S as CR2S}).}

Noting that the development of notions and results for contact round 1-surgery is parallel to the case of Weinstein's Legendrian surgery, it is natural to ask 
whether we can describe a general contact round 1-surgery on a contact 3-manifold such that Legendrian round surgery becomes a special case or class of it, similar to the case of Legendrian surgery being contact $(-1)$-surgery. 
Also, we want to realize a contact 2-surgery on the standard contact 3-sphere as a contact 2-surgery performed on a framed Legendrian knot in $\left(\s^3, \xi_{st}\right)$, 
generalizing the description of round 2-surgery on a framed knot to the standard contact 3-sphere.
Further, one expects to arrive at a theorem similar to Ding-Geiges' result.

In this article, we give the most general form of contact round 1-surgery on Legendrian links in $\left(\mathbb{S}^3, \xi_{st}\right)$ along the lines of rational contact surgery discussed by Ding-Geiges. 
In particular, we realize Legendrian round surgery as a special class of these contact round 1-surgeries as follows. 

\begin{thmintro1}
Legendrian round surgery on a Legendrian link with two components $L= L_1\cup L_2$ in $(M, \xi)$ is equivalent to a contact round 1-surgery on $L$ in $(M, \xi)$ with the identical contact round 1-surgery coefficients on each component and $\mathbb{T}^2 \times [1,2]$ with invariant contact structure on it.  
\end{thmintro1}

Further, following the description of round 2-surgery on a knot in $\s^3$ from \cite{PD-RSD}, we define a contact round 2-surgery along a Legendrian knot in the standard contact 3-sphere.
We also realize Adachi's contact round 2-surgery on an standard convex torus in $\left(\s^3, \xi_{st}\right)$ as a contact round 2-surgery on a framed Legendrian knot. This result can be stated as follows.  
\begin{thmintro2}
Adachi's contact round 2-surgery on a convex torus $T\subset (\mathbb{S}^3, \xi_{st})$ is same as a contact round 2-surgery on a Legendrian knot $K$ in $\left(\s^3, \xi_{st}\right)$ with integer round surgery coefficient, where $\dd N_{\delta}(K)= T$ for  $\delta>0$. 
Moreover, the integer $n$ depends only on the choice of the surgery meridian. 
\end{thmintro2}

As a result, we can associate framed Legendrian knots to contact round surgeries. 
It is then natural to ask whether there is a relation between contact round surgeries and the contact Dehn surgeries.
In particular, we prove that a certain set of contact round surgeries on $\left(\s^3, \xi_{st}\right)$ are equivalent to the contact $(\pm1)$-surgeries on $\left(\s^3, \xi_{st}\right)$. 
We call this result the contact bridge theorem. We state it as follows. 

\begin{thmintro3}[Contact Bridge Theorem]
The following two statements establish a bridge between a set of contact $(\pm1)$-surgery diagrams and contact round surgery diagrams of nice contact joint pairs. 
\begin{enumerate}
    \item [(1)] For a contact $(\pm1)$-surgery diagram on a Legendrian link $L_1 \cup \cdots \cup L_n \subset \s^3$ of a contact 3-manifold $(M, \xi)$, there is a round surgery diagram consisting of contact joint pairs $L= \bigcup_{i=1}^{n'} (L_{i1}\cup L_{i2})$ with round 1-surgery coefficients $k\in \mathbb{Z} $ on both components $L_{1j}$ and round 2-surgery coefficient $m \in \{\pm1\}$ on $L_{i2}$. 
    \item [(2)]Given a contact round surgery diagram of nice contact joint pairs $L= \bigcup_{i=1}^{n} (L_{1i}\cup L_{i2})$ satisfying the following conditions: 
    \begin{enumerate}
        \item The round 1-surgery coefficient is $k \in  \mathbb{Z}$ on both components, 
        \item and round 2-surgery coefficient is $m\in \{\pm1\}$ on $L_{i2}$. 
        %\item after removing thickened torus for round 2-surgery on a 2-manifold obtained by round 1-surgery, we get $\s^3\setminus int(N(L))$. We assume tight contact structure on $\s^3 \setminus int(N(L))$ is $\xi_{st}|_{\s^3 \setminus int(N(L))}$. 
    \end{enumerate}
    The Legendrian link $L= \bigcup_{i=1}^{n} (L_{1i}\cup L_{i2})$ determines a contact $(\pm1)$-surgery diagram such that for each $i,j$; $L_{ij}$ has contact surgery coefficient $m$. 
\end{enumerate}
    
\end{thmintro3}

In the existing literature, there is no Legendrian surgery presentation explored for contact round surgeries on the standard contact 3-sphere, as far as we know. 
In this article, we show that a closed connected contact 3-manifold admits a Legendrian round surgery presentation in $\left(\s^3, \xi_{st}\right)$ similar to the fact that it admits a Legendrian surgery presentation.  
We achieve this result using the contact bridge theorem and the Ding-Geiges theorem. The result can be stated as follows. 
\begin{corintro}
    Any closed connected contact 3-manifold has a Legendrian round surgery presentation in $\left(\s^3, \xi_{st}\right)$. 
\end{corintro}
   
% In particular, we establish a correspondence between certain pairs of contact round surgeries and contact $(\pm 1)$-surgeries.

This article is organized as follows. 
In Section 2, we start by recalling round surgery on links in $\s^3$ as discussed in \cite{PD-RSD}. Then, we discuss relevant notions from contact topology. We end this section with a brief overview of the contact Kirby move of type 1 from our article \cite{PD-CFKM}. 
In Section 3, we introduce contact round 1-surgery along a Legendrian link with two components in a contact 3-manifold and contact round 2-surgery along a Legendrian knot in the standard contact 3-sphere. 
We end Section $3$ with a discussion on the distinction of contact round 1-surgery and Adachi's Legendrian round surgery. 
In Section 4, we establish a correspondence between contact round surgeries and contact $(\pm1)$-surgeries and prove the contact bridge theorem. At the end, we discuss the Legendrian round surgery presentation of any closed connected contact 3-manifold in the standard contact 3-sphere .

\section*{Acknowledgement}
The first author is supported by the Prime Minister Research Fellowship-0400216, Ministry of Education, Government of India. The authors thank the referee(s) for suggesting many improvements in the content and exposition of this article. 

\section{Preliminaries}
The preliminaries are organized into three subsections. The first subsection briefly discusses the topological round surgery. In the second subsection, we 
recall relevant notions from the theory of contact 3-manifolds, namely, Legendrian knots, the theory of convex surfaces and Honda's classification of 
tight contact structures on thickened tori (see Theorem \ref{thm: ClassOfTgtCSonT2xI}). In the end, we recall the contact Kirby move of type 1 in Subsection \ref{subsec: Pre-CFKM}. The reader may skip the familiar subsections.

\subsection{Round surgery}
In \cite{Asimov}, Asimov defined round surgeries on an $n$-manifold as follows.  
\begin{defn}
Let $N$ be an $n$-manifold. 
Let $\phi : \s^1 \times \s^{k-1} \times \D^{n-k} \to N $ be an embedding. 
A round $k$-surgery on $N$ is the operation of removing the embedded region $\phi\left(\s^1 \times \s^{k-1} \times \D^{n-k}\right) $ from $N$ and gluing $\s^1 \times \D^{k} \times \s^{n-k-1}$ to get a new $n$-manifold
$$M := \overline{N \setminus \phi(\s^1 \times \s^{k-1}\times \D^{n-k})} \bigcup_{\text{id}} \left(\s^1 \times \D^{k} \times \s^{n-k-1}\right)$$
for $0\leq k\leq n-1$. 
The manifold $M$ is said to be obtained by performing a round $k$-surgery (or round surgery of index $k$) on $N$ along the embedding $\phi$. 
\end{defn}

In \cite{PD-RSD}, we defined round surgeries of indices 1 and 2 on links in $\s^3$ with rational round surgery coefficients. 
These surgery coefficients are determined by the coefficients of simple closed curves on the corresponding boundary tori around the components of a link. 
We know that any simple closed curve on a torus can be expressed with respect to the meridian $\mu$ and the chosen longitude. 
In case of a knot $K\subset \s^3$,  the tubular neighborhood $N(K)$ is diffeomorphic to a solid torus. 
On $\partial N(K)$, there is a canonical choice of longitude $\lambda$ given by a Seifert surface of $K$. 
We call $\lambda$ the canonical longitude.  
Thus, we can express any simple closed curve $c$ with respect to the meridian $\mu$ and the canonical longitude $\lambda$ as follows: 
$$ c= p\cdot \mu + q\cdot \lambda.$$ 
Moreover, the canonical longitude $\lambda$ and the meridian $\mu$ can be used to identify the boundary torus with $\R^2/\mathbb{Z}^2$ by identifying  $\mu$ to $(1,0)$ and $\lambda$ to $(0,1)$.  
Therefore, curves $\mu$ and $\lambda$ determine a coordinate system on the boundary torus, called the \textit{canonical coordinate system}.

\subsubsection{Round 1-surgery on a link with two components}
Suppose $L_1\cup L_2$ denotes a link with two components in $\s^3$. Let integers $n_1$ and $n_2$ be round 1-surgery coefficients on $L_1$ and $L_2$, respectively.  
We remove the tubular neighborhood $N(L_i)$ of each component from $\s^3$. We obtain the link-complement $$N:=\overline{\s^3\setminus(N(L_1)\cup N(L_2))}.$$ 
Now, we glue a thickened torus $\s^1 \times \D^1\times \s^1$ to $N$ by the gluing map $\phi: \dd(\s^1 \times \D^1 \times \s^1) \to \dd N$.  
The thickened torus has two boundary torus $\s^1 \times \{\pm1\}\times \s^1$, here we treat $\s^0 = \{\pm1\}$.
The boundary torus $\s^1 \times \{(-1)^i\} \times \s^1$ glues to $\partial N(L_i)$. 
The gluing map $\phi$ maps the curve $\{q\} \times \{(-1)^i\} \times \s^1$ to the meridian $\mu_i$ of $L_i$ and the curve $\s^1 \times \{(-1)^i\} \times \{p\}$ is mapped to the curve $n_i \cdot \mu_i + \lambda_i$, where $\lambda_i$ denotes the canonical longitude of the link component $L_i$ and for some $p,q \in \s^1$.

We denote by $M_{(L^{n_1}_1 \cup L^{n_2}_2)}$ the $3$-manifold obtained by performing a round $1$-surgery on $L_1 \cup L_2$ with round $1$-surgery coefficients $n_1$ and $n_2$ on $L_1$ and $L_2$, respectively.

\begin{thm}[Lemma 4, \cite{PD-RSD}]\label{thm: SameClassOfR1SD}
Suppose some integers $n_1, n_2, n^{\prime}_1$ and $n^{\prime}_2$ satisfies $n_1-n_2 = n_1^{\prime}-n^{\prime}_2$. Then,
$$M_{(L_1^{n_1}\cup L_2^{n_2})}\text{ is diffeomorphic to } M_{(L_1^{n^{\prime}_1}\cup L_2^{n^{\prime}_2})}.$$ \end{thm}

Thus, round 1-surgeries on the same link with round 1-surgery coefficients on both components to be the same integer produce the diffeomorphic 3-manifolds.

\begin{exm}\label{exm: R1SD}
A round 1-surgery diagram on the Hopf link with the round 1-surgery coefficients $-1$ on both components produces a circle bundle over torus $\s^1 \tilde{\times} \mathbb{T}^2$.
\end{exm}

We have also observed the following. 

\begin{rem}
We can not obtain a 3-torus by performing a single round 1-surgery on a Hopf link in $\s^3$. 
\end{rem}

\subsubsection{Round 2-surgery on a knot}
Suppose $K$ is a knot in $\s^3$ and $p/q\in \mathbb{Q}\cup \{\infty\}$ is the round 2-surgery coefficient. 
Let $N_d(K)$ denotes a tubular neighborhood of $K$ of some radius $d$. 
For some $0<d_1 < d_2 < d$, we remove the thickened torus $N_{d_2} \setminus N_{d_1}$ from $\s^3$ and obtain
$$N : = \overline{\s^3 \setminus \left(N_{d_2} \setminus N_{d_1}\right)}.$$
Now, we glue two copies of a solid torus, i.e., $\s^1 \times \D^2 \times \s^0$. 
For some $q \in \s^1$, we map the meridians $\{q\}\times \partial \D^2 \times \{(-1)^i\}$ to $p \cdot \mu_i + q \cdot \lambda_i$, where $\mu_i$ and $\lambda_i$ denote the canonical coordinates system on $\partial N_{d_i}$ for $i=1,2$. 

\begin{exm}\label{exm: R2SD}
A round 2-surgery on $\s^3$ producing $\s^3 \sqcup (\s^1 \times \s^2)$ can be presented by a $0$-framed unknot. 
\end{exm}

In \cite{PD-RSD}, we proved that round surgeries on links on $\s^3$ are the same as Asimov's round surgery on $\s^3$. 
The statement is as follows. 
\begin{thm}[Lemma 3 and 6, \cite{PD-RSD}]\label{Thm: RSD}
A round 1-surgery on $\s^3$ can be determined entirely by a two-component framed link in $\s^3$. And a round 2-surgery on $\mathbb{S}^3$ can be determined by a knot in $\s^3$ with a rational coefficient. 
\end{thm}

A proof of the above theorem can be found in the original article \cite{PD-RSD}.

%In Section \ref{sec: ContSurDia}, we define contact round 1-surgery on a framed Legendrian link in a contact 3-manifold $(M, \xi)$ and contact round 2-surgery on a Legendrian knot in $\left(\s^3, \xi_{st}\right)$. 

%In \cite{PD-RSD}, the authors have proved that round surgery of indices 1 and 2 can be described in a link diagram with some rational numbers as round surgery coefficients. 
The link diagrams can be thought of as an analogue of the Dehn surgery diagrams for the round surgeries.

\begin{rem}[Round 2-surgery yields a disconnected 3-manifold]
In round 2-surgery, we remove an embedded thickened torus from $\s^3$. 
It produces a 3-manifold with two components: one is the solid torus, and the other is the knot complement of the core curve of the solid torus. After gluing two copies of solid tori, we obtain two components, namely, a Lens space and a closed connected 3-manifold.
\end{rem}

\subsubsection{Joint pair of round surgeries of indices 1 and 2}
Suppose $L= L_1 \cup L_2$ is a round 1-surgery link in $\s^3$ with round 1-surgery coefficients $n_1$ on $L_1$ and $n_2$ on $L_2$. 
Since a single round 2-surgery yields a disconnected 3-manifold, we observe that we need to perform a round 2-surgery along one of the components of a round 1-surgery link to obtain a connected 3-manifold from round surgeries of indices 1 and 2. 
In particular, we perform round 2-surgery along the thickened torus parallel to the boundary $\dd N(K_2)$. We obtain a connected 3-manifold after removing this thickened torus, as discussed in the following lemma. 

\begin{lem}[Lemma 7 in \cite{PD-RSD}]\label{lem: RemovedTorusNbd}
Removing the thickened torus $\mathbb{T}^{2} \times I$ parallel to the $\dd N(K_2)$ from $M_{L^{n_1}_{1}\cup L^{n_2}_2}$ nullifies the effect of the glued thickened torus in round 1-surgery, i.e.
$$
\overline{M_{L^{n_1}_{1}\cup L^{n_2}_2} \backslash \left(\mathbb{T}^{2} \times I\right)} \cong \overline{\s^{3} \backslash \left(N(L_1)\cup N(L_2)\right)}.
$$    

\end{lem}

We fix a convention of performing round 2-surgery on the second component of the round 1-surgery link. 
This leads to an additional coefficient on $L_2$. 
%for the given round 1-surgery link $L_1 \cup L_2$. 
We call this pair of round 1-surgery and round 2-surgery a {\it joint pair}. 
Formally, we define it as follows. 

\begin{defn}
    A round 1-surgery link $L_{1} \cup L_{2}$ is said to be a {\it joint pair of round surgeries of indices 1 and 2} if one of the components of $L_{1} \cup L_{2}$ is treated as a round 2-surgery knot. 
    We denote the coefficient of the round 2-surgery knot on the top of that component next to the round 1-surgery coefficient as shown in Figure \ref{fig: JP}.
\end{defn}
\begin{figure}[ht]
    \centering
    \includegraphics[scale=0.4]{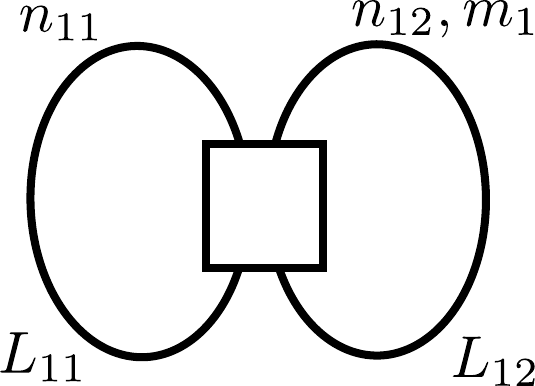}
    \caption{A diagram of a joint pair $L_{11} \cup L_{12}$ with round 1-surgery coefficient $n_{11}$ on $L_{11}$ and $n_{12}$ on $L_{12}$, and round 2-surgery coefficient $m_1$ on $L_{12}$. The box in the middle represents the linking between the components and the knotting of $L_{11}$ and $L_{12}$.}
    \label{fig: JP}
\end{figure}

We have also proved that any connected 3-manifold obtained by a sequence of round surgeries of index 1 and 2 on a link in $\s^3$ must have each round 2-surgery knot in a joint pair with a round 1-surgery link.

\begin{defn}
A round surgery diagram in $\mathbb{S}^3$ of a connected 3-manifold is a link $L=L_{1} \cup \ldots \cup L_{k}$ such that 
    \begin{enumerate}
        \item for each $1\leq i\leq n$, $L_i$ is a link with two components with an integral framing corresponding to the round 1-surgery,
        \item and for some indices $1 \leq i \leq n$, $L_i$ is a joint pair of the round surgeries of indices 1 and 2. 
    \end{enumerate}
\end{defn}

%\begin{rem}\label{rem: StdPos}We fix a convention to index a joint pair $L$ as $L_{i1} \cup L_{i2}$ such that $L_{i1} \cup L_{i2}$ is a round 1-surgery link with round 1-surgery coefficient $n_{i1}$ on $L_{i1}$ and $n_{i2}$ on $L_{i2}$, and $L_{i2}$ is also a round 2-surgery knot with round 2-surgery coefficient $m_i$, for some $i\in \mathbb{N}$. \end{rem}

\subsection{Contact $3$-manifolds}
Let $M$ be a smooth $(2n+1)$-manifold. 
\begin{defn}
    A {\it contact structure} $\xi$ on $M$ is a maximally non-integrable hyperplane field, i.e., any locally defined 1-form $\alpha$ such that ker$(\alpha)= \xi$ satisfies $\alpha \wedge (d\alpha)^n \neq 0$ at all points. The pair $(M,\xi)$ is called a contact manifold.
\end{defn}

\begin{exm}
On $\mathbb{S}^3 \subset \mathbb{R}^4$, the standard contact structure $\xi_{st}$ can be defined as the kernel of the contact 1-form $\alpha_{st}=\left( x_1dy_1-y_1dx_1+ x_2dy_2-y_2dx_2\right)$.
The contact 3-manifold $\left(\s^3, \xi_{st}\right)$ is referred to as the {\it standard contact 3-sphere}. 
\end{exm}

Two contact manifolds $(M_i, \xi_i)$, for $i=1,2$, are said to be {\it contactomorphic} if there exists a diffeomorphism $\phi: M_1 \to M_2$ such that $d\phi(\xi_1)= \xi_2$. 
Such a diffeomorphism $\phi$ is called a {\it contactomorphism} between $M_1$ and $M_2$.

A contact structure $\xi$ is said to be {\it cooriented} if the quotient bundle $TM/\xi $ is trivial. 
Moreover, for a cooriented contact structure $\xi$, there exists a global 1-form $\alpha$ such that $\xi= \text{ker}(\alpha)$.
Such a 1-form $\alpha$ is called a {\it contact form}. 
For example, the standard contact structure $\xi_{st}$ on $\s^3$ is cooriented. 
In this article, we assume that $M$ has a cooriented contact structure $\xi$ given by the kernel of a global 1-form $\alpha$ unless specified. 
We denote the cooriented contact manifold by $(M,\alpha)$. 

\subsubsection{Legendrian knots}
An embedded knot $K$ in $M$ is called {\it Legendrian knot} if it is tangent to the contact plane at each point. 
Recall that for a given Legendrian knot $K \subset (M, \xi)$ there is a tubular neighborhood $N(K)\subset M$ of $K$, which is contactomorphic to $(\s^1 \times \R^2, \text{ker}(\cos z \,  dx - \sin z\,  dy))$, where $z \in \s^1 \simeq \R/[0, 2\pi]$ and $(x,y) \in \R^2$. 
Under this contactmorphism, the spine $\s^1 \times \{0\} \subset \s^1 \times \R^2$ maps to $K \subset M$. 
We define a solid torus of radius $\delta$ as follows
\[ N_{\delta}(K)= \{(z, (x,y)) \in \s^1 \times \R^2 | \,  x^2 +y^2 = \delta^2\}. \]
Suppose $K$ is a null-homologous knot in $M$. 
%In particular, if $M= \s^3$, then there is the canonical longitude on the boundary torus $\Sigma$ such that $\dd \Sigma= K$.
%On $K$, there is a vector field transverse to $K $ but tangent to $\Sigma$. This vector field is called the {\it framing} of $K$. 
%On $\dd N_{\delta}(K)$, the push off of $K$ along this vector field in $N_{\delta}(K)$ gives an isotopic curve $\lambda$, called the {\it canonical longitude}.
The twisting of contact planes along $K$ with respect to the canonical longitude $\lambda$ defines an invariant of the Legendrian knot $K$, called {\it Thurston--Bennequin invariant}.
Moreover, on $\partial N_{\delta}(K)$, the contact structure also induces a longitude called the {\it contact longitude} $\lambda_c$ such that 
        $$\lambda_c = tb(K) \cdot \mu + \lambda; $$ 
        where $tb(K)$ is the Thurston--Bennequin invariant of $K$.

We are interested in studying contact round surgery on a Legendrian link in $\left(\s^3, \xi_{st}\right)$. 
%For that, we recall the front projection of a Legendrian knot in $(\mathbb{R}^3,)$ on to $(y,z)$-plane. 
Since any Legendrian knot in $\left(\s^3, \xi_{st}\right)$ can be realized as a Legendrian knot in $\left(\R^3, \xi_{st}^{\prime}(:=\text{ker}(dz+xdy))\right)$, we project a Legendrian knot onto $(y,z)$-plane in $\R^3$ to draw the contact surgery diagrams.  
These projections of the Legendrian knots are called the {\it front projections}. 

Suppose $\gamma: \s^1 \to \mathbb{R}^3$ is an embedding given by $\gamma(s)= (x(s), y(s), z(s))$ such that $\gamma(\s^1)$ is a Legendrian knot $K$.

\begin{defn}
    The {\it front projection} of a parametrised curve $\gamma(s)$ in $(\mathbb{R}^3, \xi'_{st})$ is the curve $$\gamma_F(s)= (y(s), z(s)).$$ 
    %its {\it Lagrangian projection }is the curve $\gamma_L(s)=(x(s), y(s))$. 
\end{defn}

\subsubsection{Tight and Overtwisted contact structures.}
An embedded disk $D$ in a contact 3-manifold $(M,\xi)$ is an {\it overtwisted disk},  if it has a Legendrian boundary whose surface framing coincides with the contact framing, and in the interior, there is a point with coinciding contact plane and tangent plane to the disk. 
A contact structure on a 3-manifold is said to be {\it overtwisted} if it contains an overtwisted disk. 
For an example, the kernel of the 1-form $\A_{ot}= \cos r dz + r\sin r d\phi$ on $\R^3$, formulated in polar coordinates $(0\leq r, 0\leq \phi< 2\pi)$ on the $(x,y)$-plane, defines an overtwisted contact structure $\xi_{ot}$. 
In particular, the disk $D=\{z=0, r \leq \pi\}$ is an overtwisted disk on $(\R^3, \xi_{ot})$.

\subsubsection{Convex surfaces}
\begin{defn}
    A vector field $X$ on $(M,\xi)$ is called a {\it contact vector field} if its flow preserves $\xi$. 
\end{defn}

\begin{exm}
    The radial vector field $X:= x\partial_x + y\partial_y$ is a contact vector field in the open tubular neighborhood $\s^1 \times \mathbb{R}^2$ of a Legendrian knot $K$.
\end{exm}

We recall some terminology from the theory of convex surfaces. 
For the details, the reader may refer to \cite{Book-HG} or \cite{Giroux-Convex91}. 

\begin{defn}
    An embedded closed surface $\Sigma$ in $(M,\xi)$ is called {\it convex surface} if it admits a transverse contact vector field in its neighborhood. 
\end{defn}

The boundary $\partial N_{\delta}(K)$ is transverse to the radial vector field X. Hence, it is a convex surface in $(\s^1 \times \mathbb{R}^2, \text{ker}(\cos z \,dx-\sin z\, dy))$. 

Given a convex surface $\Sigma$, the set $\Gamma_{\Sigma}$ of all points where the contact vector field is tangent to the contact structure is called the dividing set of $\Sigma$. In the above example, $\Gamma_{\partial N_{\delta}(K)}$ is the set of points $w \in \partial N_{\delta}(K)$ such that $X(w) \in \text{ker}(\cos z \,dx-\sin  z\, dy))$. 

The following theorem is an important tool to prove the overtwistedness of a contact structure. 

\begin{thm}[Giroux's criterion, \cite{Giroux-Convex91}]\label{thm: Giroux Criterion}
Let $\Sigma$ be convex surface in $(M, \xi)$. A vertically invariant neighborhood of $\Sigma$ is tight if and only if $\Sigma \neq \s^2$ and $\Gamma_{\Sigma}$ contains no contractible curves or $\Sigma = \s^2$ and $\Gamma_{\Sigma}$ is connected. 
\end{thm}

A convex torus $T$ is said to be in {\it standard form} if, under some identification of $T$ with $\mathbb{R}^2/\mathbb{Z}^2$, 
\begin{enumerate}
    \item the dividing curves $\Gamma_T$ consist of $2n$ parallel homotopically essential curves of slope 0, 
    \item and the Legendrian rulings with coordinates $(x,y)$ are given by $y=rx+b$, where $r\neq 0$ is fixed, and $b$ varies in a family, with tangencies $y=\frac{k}{2n}$, $k= 1, \ldots, 2n$.  
\end{enumerate}

Notice that in round 1-surgeries on framed links, we glue a thickened torus along two tori. 
In order to extend the contact structure from the complement of the link to the resultant manifold, we need to know all the possible tight contact structures on the thickened torus.      
Therefore, we recall Honda's classification theorem for tight contact structures on $\mathbb{T}^2\times I$ (cf. \cite{KH}).
For the statement of the theorem, we need the notion of twisting $\phi_I$ in the $I$-direction, minimal twisting in the $I$-direction, and nonrotativity in the $I$-direction. For that, the reader may refer to Section 2 of \cite{KH}. 

Recall that the set of dividing curves of a given convex torus $T$ in $\mathbb{T}^2 \times I$ is, up to isotopy, determined by the number $\# \Gamma_T$ of these dividing curves and their slope $s(T)$, defined by the property that each curve is isotopic to a linear curve of slope $s(T)$ in $T \simeq \mathbb{R}^2 / \mathbb{Z}^2$. This information about the dividing curves on the boundary torus is called the boundary data. 

We may normalize the boundary slopes by changing the coordinate system and assume dividing curves with slope $-\frac{p}{q}$, where $p \geq q>0,(p, q)=1$, and $T_0$ has slope $-1$. We denote $T_a = \mathbb{T}^2\times \{a\} $. For this boundary data, we have the following.

\begin{thm}[Classification of tight contact structuress on $\mathbb{T}^2 \times I$, \cite{KH}]\label{thm: ClassOfTgtCSonT2xI}
Consider $\mathbb{T}^2 \times I$ with convex boundary, and assume, after normalizing via $S L(2, \mathbb{Z})$, that $\Gamma_{T_1}$ has slope $-\frac{p}{q}$, and $\Gamma_{T_0}$ has slope $-1$ . Assume we fix a characteristic foliation on $T_0$ and $T_1$ with these dividing curves. Then, up to an isotopy which fixes the boundary, we have the following classification:
\begin{enumerate}
    \item Assume either (a) $-\frac{p}{q}<-1$ or (b) $-\frac{p}{q}=-1$ and $\phi_I>0$. Then there exists a unique factorization $\mathbb{T}^2 \times I=\left(\mathbb{T}^2 \times\left[0, \frac{1}{3}\right]\right) \cup\left(\mathbb{T}^2 \times\left[\frac{1}{3}, \frac{2}{3}\right]\right) \cup$ $(\mathbb{T}^2 \times\left[\frac{2}{3}, 1\right])$, where (1) $T_{\frac{i}{3}}, i=0,1,2,3$, are convex, (2) $\left(\mathbb{T}^2 \times\left[0, \frac{1}{3}\right]\right)$ and $(\mathbb{T}^2 \times\left[\frac{2}{3}, 1\right])$ are nonrotative, (3) $\# \Gamma_{T_{\frac{1}{3}}}=\# \Gamma_{T_{\frac{2}{3}}}=2$, and (4) $T_{\frac{1}{3}}$ and $T_{\frac{2}{3}}$ have fixed characteristic foliations which are adapted to $\Gamma_{T_{\frac{1}{3}}}$ and $\Gamma_{T_{\frac{2}{3}}}$.
    \item  Assume $-\frac{p}{q}<-1$ and $\# \Gamma_{T_0}=\# \Gamma_{T_1}=2$.
    \begin{enumerate}
    \item  There exist exactly $\left|\left(r_0+1\right)\left(r_1+1\right) \cdots\left(r_{k-1}+1\right)\left(r_k\right)\right|$ tight contact structures with $\phi_I=0$. Here, $r_0, \ldots, r_k$ are the coefficients of the continued fraction expansion of $-\frac{p}{q}$, and $-\frac{p}{q}<-1$.
        \item There exist exactly 2 tight contact structures with $\phi_I=n$, for each $n \in \mathbb{Z}^{+}$.
    \end{enumerate}

\item Assume $-\frac{p}{q}=-1$ and $\# \Gamma_{T_0}=\# \Gamma_{T_1}=2$. Then there exist exactly 2 tight contact structures with $\phi_I=n$, for each $n \in \mathbb{Z}^{+}$.
    \item Assume $-\frac{p}{q}=-1$ and $\# \Gamma_{T_0}=2 n_0, \# \Gamma_{T_1}=2 n_1$. Then the non-rotative tight contact structures are in 1-1 correspondence with $\mathcal{G}$, the set of all possible (isotopy classes of) configurations of arcs on an annulus $A= \s^1 \times I$ with markings $\sigma_i \subset \s^1 \times\{i\}, i=0,1$, which satisfy the following:
    \begin{enumerate}
        \item  $\left|\sigma_i\right|=2 n_i, i=0,1$, where $|\cdot|$ denotes cardinality.
        \item  Every point of $\sigma_0 \cup \sigma_1$ is precisely one endpoint of one arc.
        \item There exist at least two arcs which begin on $\sigma_0$ and end on $\sigma_1$.
        \item There are no closed curves.
    \end{enumerate}
\end{enumerate}
\end{thm}

The classification of non-rotative tight contact structures essentially uses the following proposition from \cite{KH}.
In Section \ref{sec: ContSurDia}, we mention the holonomy map defined in the following lemma. 

\begin{lem}[Proposition 4.9 in \cite{KH}]\label{lem: NonRotTightContStr}
    Let $\Gamma_{T_i}, i=0,1$, satisfy $\# \Gamma_{T_i}=2$ and $s_0=s_1=-1$. 
    Then there exists a holonomy map $k: \pi_0\left(\right.$Tight$^{\min}\left(\mathbb{T}^2 \times I, \Gamma_{T_1} \cup\right.$ $\left.\left.\Gamma_{T_2}\right)\right) \rightarrow \mathbb{Z}$ which is bijective.
    Here, $Tight^{\min }\left(\mathbb{T}^2 \times I, \Gamma_{T_1} \cup \right.\left.\Gamma_{T_2}\right)$ denotes the set of all tight contact structures on $\mathbb{T}^2 \times I$ with minimal twisting in the $I$-direction and dividing curves $\Gamma_{T_1} \cup \Gamma_{T_2}$ on the boundary tori $T_1 \cup T_2$. 
\end{lem}

We will use the contact Kirby move of type 1 (cf. \cite{PD-CFKM}) to define the contact joint pairs. Contact joint pair of contact round surgeries of indices 1 and 2 will be used to establish a correspondence between contact $(\pm1)$-surgery and round surgery diagrams that consist of only contact joint pairs. 

\subsection{Contact Kirby move of type 1}\label{subsec: Pre-CFKM}
Contact Dehn surgery is an operation on the contact 3-manifold to obtain a new contact 3-manifold.  
Intuitively, one wants to perform a Dehn surgery on a Legendrian knot with respect to the contact longitude $\lambda_c$.
It means that the surgery coefficient $r=\frac{p}{q}$ is given with respect to the contact longitude. This surgery coefficient is called the \emph{contact surgery coefficient}. 

For surgery coefficient $1/q$ and $q\in \mathbb{Z}$, the contact structure on the resultant contact 3-manifold is completely determined by the Legendrian knot and the coefficient. 
However, for other coefficients, the contact structure on the resultant contact 3-manifold also depends on the choice of a tight contact structure on the glued solid torus. 
In the case of integer coefficients, there are at most two choices for contact structures on the glued solid torus. Due to Theorem~\ref{thm: LegPresnthm}, each choice of the contact structure corresponds to a unique contact $(\pm1)$-surgery presentation.

We consider a contact $(m+1)$-surgery diagram on a Legendrian unknot $K$ with $tb(K) = -m$, where all stabilizations are performed on the right, as depicted on the left in Figure~\ref{fig:NewDiagram}. 
This surgery diagram corresponds to at most two contact structures on the resulting 3--manifold, each of which admits a description via a contact $(\pm 1)$-surgery presentation. 
By the classification theorem for integral cosmetic contact surgeries (cf.~Lemma~5.4.4 in~\cite{MK-thesis} and Corollary~3.5 in~\cite{EKO}), the contact $(\pm 1)$-surgery on the Legendrian link shown on the right in Figure~\ref{fig:NewDiagram} represents $(S^3, \xi_{st})$. 
Moreover, every surgery diagram of the standard contact structure $\xi_{st}$ on $S^3$ along a Legendrian unknot in $(S^3, \xi_{st})$ arises in this manner.

\begin{figure}[ht]
    \centering
    \includegraphics[scale=0.5]{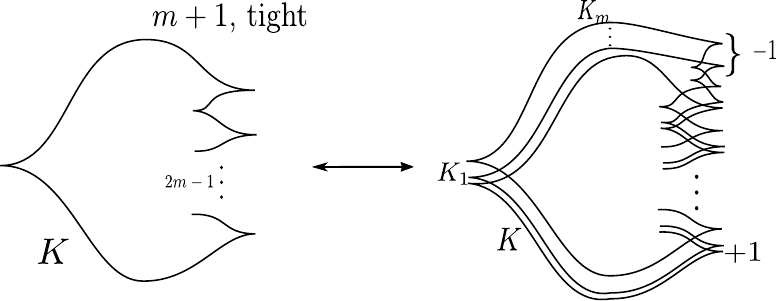}
    \caption{New contact surgery diagram that shows corresponding contact $(\pm1)$-surgery presentation produce $(\mathbb{S}^3, \xi_{st})$.}
    \label{fig:NewDiagram}
\end{figure}

To specify $(\pm1)$-surgery presentation by means of a surgery diagram on a Legendrian unknot, 
we label the Legendrian unknot $K$ with the word \textit{tight} above it, as illustrated in Figure~\ref{fig:NewDiagram}. 
We denote by $\mathcal{C}$ the collection of all integral contact cosmetic surgeries on a Legendrian unknot whose $(\pm1)$-surgery presentation corresponds to $\left(S^3, \xi_{st}\right)$.

A potential contact analogue of the Kirby move of type~1 is the \textit{addition} (or \textit{deletion}) of a contact surgery diagram $K$ to (or from) a given integral contact surgery diagram, where $K \in \mathcal{C}$.

In Figure~\ref{fig:NewDiagram}, we depict a presentation of $(\pm1)$-surgery  consisting of a contact $(+1)$-surgery on one component $K$, which is identical to the Legendrian unknot $K$ shown on the left, together with contact $(-1)$-surgeries on $m$ components $K_1, \ldots, K_m$. 
Each $K_i$ is obtained from $K$ by performing a single stabilization on the right.

\section{Contact Round Surgeries and Their Diagrams}\label{sec: ContSurDia}
In this section, we define \emph{contact round $1$-surgery} on a framed Legendrian link within a contact $3$-manifold and \emph{contact round $2$-surgery} on a Legendrian knot in the standard contact $3$-sphere, ensuring that the resulting $3$-manifold inherits a contact structure. 
Furthermore, we show that \emph{Adachi’s Legendrian round surgery} can be viewed as a contact round $1$-surgery on a Legendrian link, and that \emph{Adachi’s contact round $2$-surgery} on a convex torus in $(S^3, \xi_{st})$ corresponds to performing contact round $2$-surgery on a framed Legendrian knot in the standard contact $3$-sphere.

\subsection{Contact round 1-surgery}\label{subsec: CR1S}
Suppose $L_1 \cup L_2$ is a Legendrian link in a contact 3-manifold $(M, \xi)$. 
We wish to define a contact round 1-surgery on this link. 
Suppose $n_1$ and $n_2$ are the round 1-surgery coefficients on $L_1$ and $L_2$ with respect to the contact longitudes $\lambda_{c_1}$ and $\lambda_{c_2}$ respectively. 
In round 1-surgery, we first remove the interiors of the standard neighborhoods $N_{\delta}(L_1)\cup N_{\delta}(L_2)$ from $(M, \xi)$.
We obtain $N:= M \setminus \{ int(N_{\delta}(L_1))\cup int(N_{\delta}(L_2))\} $. 
Clearly, $\dd N = \dd N_{\delta}(L_1) \cup \dd N_{\delta}(L_2)$ and each boundary torus $\dd N_{\delta}(L_i)$ is a convex torus with two dividing curves isotopic to the contact longitude.

Now, we glue a thickened torus $\mathbb{T}^2 \times [1,2]$ to $N$ by the gluing map $\phi: \dd(\mathbb{T}^2\times [1,2]) \to \dd N$.
By $x$ and $y$ we denote representative curves of the generating classes of the first homology of the thickened torus. Our choice of representative curves, within the generating class, is unique up to isotopy of simple closed curves in $\mathbb{T}^2$.
We identify $\mathbb{T}^2 \times \{t\}$ with $(\mathbb{R}^2/\mathbb{Z}^2)\times \{t\}$ such that $x \mapsto (1,0)$ and $y\mapsto (0,1)$, for $t\in [1,2]$.
Suppose $\mathbb{T}^2 \times [1,2]$ has a tight contact structure such that the boundary tori are convex with two dividing curves of slopes $\frac{-1}{n_j}$ on the boundary torus $\mathbb{T}^2 \times \{j\}$. 

By the description of the (topological) round 1-surgery on a framed link, the gluing map $\phi$ maps the curve $x$ to $\mu_j$ and $y$ to $n_j\cdot \mu_j+ \lambda_{c_j}$ on $\mathbb{T}^2 \times \{j\}$, for each $j=1,2$.

In particular, the dividing curves map to the dividing curves under the gluing map $\phi$. 
Since the germ of the contact structure near a convex surface is completely 
determined in terms of the isotopy class of dividing curves, we can extend the contact structure $\xi$ to a contact structure $\zeta$ on the resulting 3-manifold $M$ obtained after the round 1-surgery on $L$.
By Theorem \ref{thm: ClassOfTgtCSonT2xI}, $\mathbb{T}^2 \times [1,2]$ has many tight contact structures satisfying given boundary conditions.
Therefore, the contact structure $\zeta$ on $M$ depends not only on the Legendrian link $L$ and round 1-surgery coefficients but also on the choice of the tight contact structure on $\mathbb{T}^2 \times [1,2]$. 

We say contact 3-manifold $(M, \zeta)$ is obtained by performing contact round 1-surgery on $\left(L_1 \cup L_2\right) \subset (M, \xi)$ with round 1-surgery coefficients $n_1$ on $L_1$ and $n_2$ on $L_2$. 
Moreover, the contact structure $\zeta$ depends on the framed link and choice of the tight contact structure on $\mathbb{T}^2 \times [1,2]$.

When $M$ is the standard contact $3$-sphere, we can represent the Legendrian link in $(\mathbb{R}^3, \xi_{st}^{\prime})$ by its front projection to illustrate the diagram corresponding to contact round $1$-surgery.

While drawing the diagram, we need to mention the chosen tight contact structure on the glued thickened torus to completely describe the resultant contact 3-manifold.
For example, see Figure \ref{fig: ExampleContR1Surgery}, we obtain a circle bundle over a torus with a tight contact structure on it by performing contact round 1-surgery on the Legendrian Hopf link with contact round 1-surgery coefficient $0$ on both components and gluing an $I$-invariant neighborhood of the standard convex torus. 

\begin{figure}[ht]
    \centering
    \includegraphics[scale=0.3]{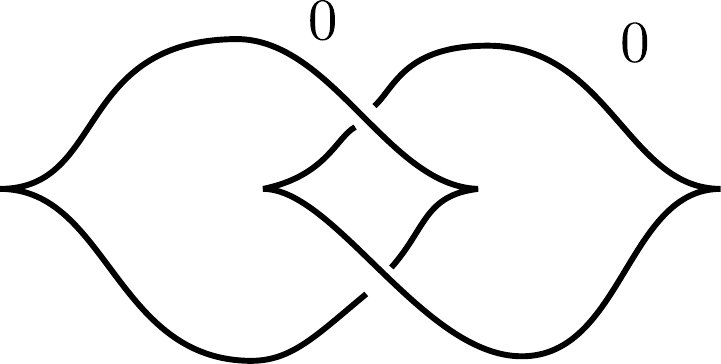}
    \caption{Contact round 1-surgery presentation with gluing $I$-invariant tubular neighborhood of the standard convex torus corresponds to a circle bundle over a torus with a tight contact structure.}
    \label{fig: ExampleContR1Surgery}
\end{figure}

The contact round 1-surgery coefficient $0$ equals the round 1-surgery coefficient $-1$ on both components of the Hopf link. 
By Example~\ref{exm: R1SD}, the round 1-surgery diagram on the Hopf link with the $-1$ coefficients on both components produces a circle bundle over a torus. 
Moreover, this contact round 1-surgery is equivalent to Adachi's Legendrian round surgery by Theorem~\ref{thm: JCR1StoCR1S}. 
In particular, the resultant contact structure on the circle bundle over a torus is symplectically fillable; hence, the resultant contact structure on the circle bundle over a torus is tight.

\iffalse 
\begin{rem}
    The above definition of contact round 1-surgery is independent of the base contact 3-manifold $\left(\s^3, \xi_{st}\right)$. 
    Thus, we can take any cooriented contact 3-manifold $(M, \xi)$ in place of $\left(\s^3, \xi_{st}\right)$ and define the contact round 1-surgery on $(M, \xi)$ as above. 
\end{rem}\fi

\subsection{Contact round 2-surgery}\label{subsec: CR2S}
In round 2-surgery, we remove a thickened torus and glue two solid tori. 
From Theorem~\ref{Thm: RSD}, we know that a round 2-surgery on a thickened torus in $\s^3$ is determined by a knot $K$ with round 2-surgery coefficient $\frac{p}{q}$. 
In particular, the thickened torus $\mathbb{T}^2 \times [1,2]$ embeds as $ N_{\delta_2}(K)\setminus int(N_{\delta_1}(K))$. 
Suppose $K$ is a Legendrian knot in $\s^3$ with a round 2-surgery coefficient $\frac{p}{q}$ with respect to the contact longitude $\lambda_c$. 
After removing this thickened torus, we obtain $N= \overline{\s^3 \setminus \{N_{\delta_2}(K)\setminus int(N_{\delta_1}(K))}\}$$= (\s^3 \setminus int(N_{\delta_2}(K))) \sqcup N_{\delta_1}(K)$, i.e. a disjoint union of the knot exterior and the tubular neighborhood of $K$. 
Notice that glueing a solid torus to the knot exterior, producing a 3-manifold that admits a contact structure, is the same as performing contact Dehn surgery on $K$ with contact surgery coefficient $\frac{p}{q}$.
Moreover, the gluing of a solid torus to the tubular neighborhood $N_{\delta_1}(K)$ of $K$ produces a lens space $L(a,b)$ for some $a, b \in \mathbb{Z}$.
Each boundary torus of the boundary $\dd N = \dd N_{\delta_2}(K) \cup \dd N_{\delta_1}(K)$ is convex with two dividing curves parallel to the contact longitude.
Now, we glue one solid torus along each boundary component, two in total.
We denote a solid torus by $T_j$ if it glues to $\dd N_{\delta_j}(K)$ for $j=1,2$. 
On $\dd T_j$, suppose $m_j$ and $l_j$ denote the meridian and a longitude respectively.  
Suppose $p\neq 0$. Then, the preimage of the dividing curve (isotopic to $\lambda_c$) has a non-zero slope. 
Thus, we can choose $T_j$ with a tight contact structure by classification of tight contact structures of $\s^1 \times \D^2$ in \cite{KH} satisfying the slope condition and extend $\xi_{st}$ to a contact structure $\zeta$ on the resultant 3-manifold $M$.

If $p= 0$ (and $q=\pm 1$), then the boundary slope of $\dd N_{\delta}(K)$ is zero. 
We glue in a solid torus with a contact structure defined as follows. 

Let $\zeta_0$ be the contact structure on $\s^1\times \D^2$ defined by the smooth 1-form $\beta_0= h_1(r)d\theta+ h_2(r) d\phi$, where  $\theta$ is $\s^1-$coordinate and $0\leq r, \,  \phi\in [0,2\pi]$ are the polar coordinates on $\D^2$. 
Here we impose the following general conditions on the smooth plane curve $r \to \gamma(r)= (h_1(\theta), h_2(\theta))$: 
$h_1(r)= -1$ and $h_2(r)= -r^2$, for near $r=0 $, and the position vectors $\gamma(r) \text{ and } \gamma'(r)$ are not parallel to each other. 
For the $0$-surgery,  we take $h_1(1)=1 $ and $h_2(1)= 0$. 
We may perturb this torus into a convex surface with $\# \Gamma_{T^2}=2$ and slope $0$.

The solid torus $(T, \z_0)$ is glued to $M \setminus N_{\delta}(K)$ using the attaching map corresponding to $p=0$.
%, i.e. the map sending meridian to contact longitude $\lambda_c$.

We say contact 3-manifold $(M, \zeta) \sqcup (L(a, b), \chi)$ is obtained by performing contact round 2-surgery on $K \subset \left(\s^3, \xi_{st}\right)$ with round 2-surgery coefficient $\frac{p}{q}$ on $K$. 
Moreover, the contact structure $\zeta$ and $\chi$ depends on $K$, coefficient $\frac{p}{q}$ and choice of the tight contact structure on $\mathbb{S}^1 \times \D^2$.
%There is a unique tight contact structure on $\s^1 \times \D^2$ when the coefficient is $1/q$.
%In this case, the contact structures $\zeta$ and $\chi$ only depend on the knot $K$ and its surgery coefficient $1/q$. 

\begin{figure}[ht]
    \centering
    \includegraphics[scale=0.5]{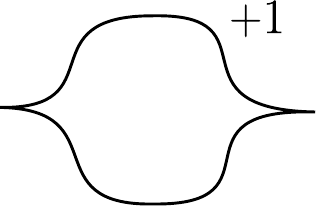}
    \caption{Contact round 2-surgery presentation of $(\s^1 \times \s^2, \xi_{tight}) \sqcup \left(\s^3, \xi_{st}\right)$.}
    \label{fig: ExamContR2Surgery}
\end{figure}

For example, we can present $(\s^1 \times \s^2, \xi_{tight}) \sqcup \left(\s^3, \xi_{st}\right)$ via a single Legendrian unknot with contact round 2-surgery coefficient $+1$ (see Figure \ref{fig: ExamContR2Surgery}). 
We observe that this contact round 2-surgery can be realized as a pair of contact Dehn surgeries on two $\left(\s^3, \xi_{st}\right)$, namely, the contact $(+1)$-surgery on $\left(\s^3, \xi_{st}\right)$ and contact trivial surgery on $\left(\s^3, \xi_{st}\right)$. 
The first contact Dehn surgery yields $(\s^1 \times \s^2, \xi_{tight})$ and other one yields $\left(\s^3, \xi_{st}\right)$.
It implies there is a unique tight contact structure on the glued solid tori while performing round 2-surgery. 
Thus, we do not need to explicitly mention the contact structure on the glued solid torus while drawing the diagram.

\begin{rem}
It is easy to see that a (topological) round 2-surgery is complementary to the round 1-surgery, i.e., there exists a round surgery of index 1 cancelling the effect of a round surgery of index 2 and vice versa. 
For example, we can perform a topological round 1-surgery along $\s^1 \times \{p\} \subset \s^1 \times \s^2$ and an unknot $U \subset \s^3$, to recover $\s^3$ again.
Thus, it is natural to ask whether the contact round surgery of index 2 is complementary to the contact round surgery of index 1. It turns out that it is not always the case that the contact round surgery of indices 1 and 2 are complementary operations. To this end, we consider the following argument. 
We perform a contact round 2-surgery on $\left(\mathbb{T}^2 \times [-1,1], \zeta\right) \subset (M, \xi)$ and obtain 
$$ (M', \xi')= \left\{(M, \xi) \setminus \left(\mathbb{T}^2 \times [-1,1], \zeta\right)\right\} \cup  \left(T_j, \xi_j\right). $$
Here, $\left(T_j, \xi_j\right)$ denotes a solid torus with contact structure $\xi_j$ that glues along the boundary torus $\mathbb{T}^2 \times \{j\} \subset \mathbb{T}^2 \times [-1,1]$. 
Now, if the solid torus $\left(T_{j}, \xi_{j}\right)$ can be identified as a standard tubular neighborhood of a Legendrian knot $L_j$ in $\left(M', \xi'\right)$, then we can perform a contact round 1-surgery on $\left(M', 
\xi'\right)$ along the solid torus $\left(T_{j}, \xi_{j}\right)$ and glue $\left(\mathbb{T}^2 \times [-1,1], \zeta\right)$ to recover $(M, \xi)$ again. Otherwise, it is not possible to perform contact round surgery of index $1$ to undo the effect of contact round $2$-surgery.
%Thus, for a contact round 2-surgery, there always exists a complementary contact round 1-surgery.}
One more difficulty lies in the fact that it is not immediate to find such embedded solid tori to perform a complementary contact round 1-surgery on a contact 3-manifold obtained by a contact round surgery of index 2.
The same difficulty lies in the case of performing a complementary contact round 2-surgery on a contact 3-manifold obtained by a contact round surgery of index 1.   
\end{rem}

%For a contact round 2-surgery on $\left(M, \xi\right)$, there exist a contact round 1-surgery along some Legendrian link in $\left(M', \xi'\right)$. 
%In the above example, we can a perform a topological round 1-surgery along $\s^1 \times \{p\} \subset \s^1 \times \s^2$ and an unknot $U \subset \s^3$, and obtain $\s^3$ back.
%In order to realize the effect in contact category, we must extend the contact struture to $\s^3$ to get tight contact structure on it. 
%In the topological category, it is clear that round surgeries are complimentary to each other. However, in the contact category, it is not immediate to this unless we specify the convex tori, along which the gluing take place for round surgeries.

\subsection{Realisation of Adachi's Legendrian round surgery on $\left(\s^3, \xi_{st}\right)$ as a special class of contact round 1-surgery}
In \cite{Jiro}, Adachi introduced Legendrian round surgery on $(M, \xi)$ as the attachment of the symplectic round handle to the convex end of the symplectization $(M\times [0,1], d(e^t\alpha))$, where $\xi=\text{ker}(\alpha)$. 
In particular, Legendrian round surgery is a round version of the Weinstein surgery on a contact 3-manifold. 
Adachi defined the Legendrian round surgery as follows. 
\begin{enumerate}
    \item Take the standard tubular neighborhood $N(L_i)$ of each Legendrian knot $L_i$, $i=1,2$, so that $N(L_1) \cap N(L_2)= \emptyset$. 
    Then remove the interiors $int(N(L_i))\subset(M, \xi_i)$, i=1,2. 
    \item Reglue the invariant tubular neighborhood $\mathbb{T}^2 \times [-\epsilon, \epsilon]$ of the standard convex torus with a fixed meridian so that the meridian and the dividing curves of $\dd N(L_1)$, $\dd N(L_2)$ and $\mathbb{T}^2 \times \{\pm \epsilon\}$ agree respectively. 
\end{enumerate}

\begin{thm}\label{thm: JCR1StoCR1S}
Legendrian round surgery on a Legendrian link with two components $L= L_1\cup L_2$ in $(M, \xi)$ is equivalent to a contact round 1-surgery on $L$ in $(M, \xi)$ with the identical contact round 1-surgery coefficients on each component and $\mathbb{T}^2 \times [1,2]$ with invariant contact structure on it.  
\end{thm}

\begin{proof} 
While performing Legendrian round surgery, we remove the standard tubular neighborhood 
$N(L_1) \cup N(L_2)$ of a Legendrian link $L_1 \cup L_2$ and glue in an invariant 
tubular neighborhood of a standard convex torus. 
Let $x$ and $y$ denote the meridian and longitude curves on 
$\mathbb{T}^2 \times \{1,2\}$, which form the boundary of the thickened torus used in the gluing, such that the boundary slopes are $-1$ on both boundary components. Consequently, the dividing curves 
$c_j$ can be expressed as 
\[
    c_j = y-x, \quad \text{for each } j = 1, 2.
\]
During the gluing process, we identify the dividing curves and meridians of $\mathbb{T}^2 \times \{j\}$ with the corresponding dividing curves 
$\lambda_{c_j}$ and meridians $\mu_j$ of $\partial N(L_j)$, for $j = 1, 2$, as follows:
\[
\begin{aligned}
c_j &\mapsto \lambda_{c_j}, & \quad x &\mapsto \mu_j,\\
\Rightarrow\; y - x &\mapsto \lambda_{c_j}, & \quad x &\mapsto \mu_j,\\
\Rightarrow\; y &\mapsto \lambda_{c_j} + \mu_j, & \quad x &\mapsto \mu_j.
\end{aligned}
\]
In our framework, this operation corresponds to a contact round 1-surgery 
with coefficient $+1$ on both components $L_1$ and $L_2$, using a glued invariant 
tubular neighborhood of a standard convex torus.

By Theorem~\ref{thm: SameClassOfR1SD}, a (topological) round 1-surgery on 
$L_1 \cup L_2$ with identical contact round 1-surgery coefficients on both 
components yield the same topological 3-manifold. 
Since, in each case, we glue the invariant tubular neighborhood of the standard 
convex torus, the resulting contact structures on the 3-manifold coincide.
Therefore, performing a {contact round 1-surgery} on $L_1 \cup L_2$ with 
identical contact round 1-surgery coefficients on both components—and gluing in 
an invariant tubular neighborhood of the standard convex torus—is equivalent to 
performing a {Legendrian round surgery} on $L_1 \cup L_2$.
\end{proof}

\subsubsection{A remark on the difference between the contact round 1-surgery and Adachi's Legendrian round surgery.}

In Theorem~\ref{thm: JCR1StoCR1S}, we have realized Adachi's contact round 1-surgery (or, Legendrian round surgery) as a special case of contact round 1-surgery. 
In particular, Legendrian round surgery on the standard contact 3-sphere produces a tight contact structure on the resultant 3-manifold.
Here, we give an example of a contact round 1-surgery on $\left(\s^3,\xi_{st}\right)$ that produces an overtwisted contact structure on the resultant 3-manifold. 
For example, consider a Hopf link $L$ with contact round 1-surgery coefficient $-1$ on each component as shown in Figure \ref{fig: ExampleContR1Surgery}. 
We perform a contact round 1-surgery on it with a thickened torus having a rotative tight contact structure $\xi_{2m}^{+}$ (see Lemmas 5.2 and 5.3 in \cite{KH} for the descriptions of $\xi_{2m}^{+}$). 
The above surgery yields a circle bundle over a torus with an overtwisted contact structure.  
To see this, consider the following. 

After removing the interior of the standard tubular neighborhoods of the components, we get $(\s^3\setminus int(N(L)), \xi_{st}|_{\s^3\setminus int(N(L))})$. Since $L$ is a Hopf link, $\s^3\setminus int(N(L)) \cong \mathbb{T}^2 \times [0,1]$.
The boundary slopes are $-1$ with respect to the canonical coordinates of the link components. 
We need to express boundary slopes in only one coordinate system. Since $L$ is a Hopf link, the meridian of one component maps to the longitude of the other and vice versa. Thus, fixing one canonical coordinate system as the coordinates of the $\mathbb{T}^2\times [0,1]$ is sufficient. As a result, the tight contact structure $\xi_{st}|_{\s^3\setminus int(N(L))}= \xi_{st}|_{\mathbb{T}^2 \times [0,1]}$ with boundary slopes $-1$. Moreover, since there is no Giroux torsion in $\xi_{st}$, the $ \xi_{st}|_{\mathbb{T}^2 \times [0,1]}$ is minimal twisting non-rotative.

\begin{figure}[ht]
    \centering
    \includegraphics[scale=0.4]{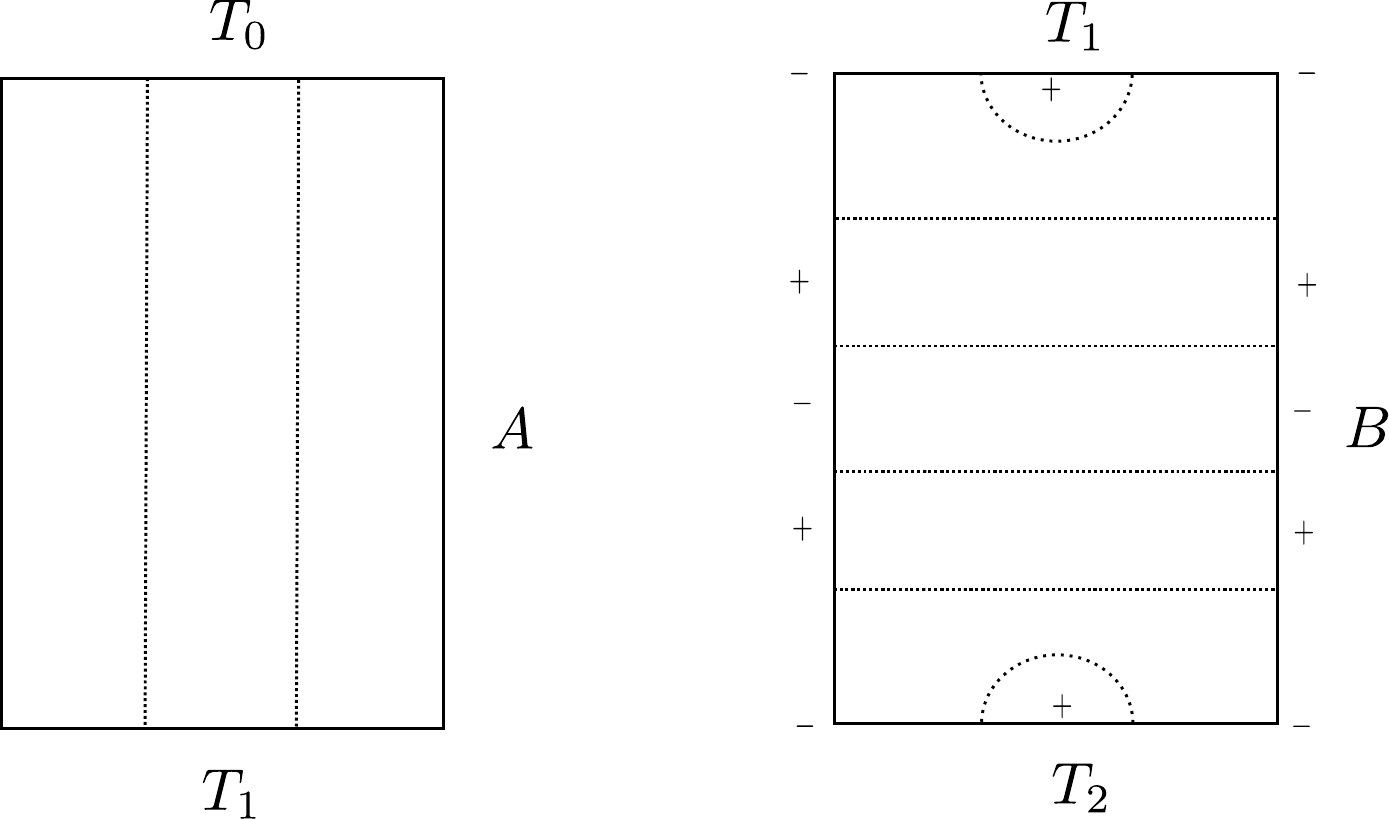}
    \caption{One of the possible configurations of the dividing curves on the annulus A and B in black dotted curves. In each rectangle, the left and right sides are identified.}
    \label{fig: gluingAnnuliToOD}
\end{figure}

By Giroux's flexibility theorem \cite{Giroux-Convex91}, without loss of generality, we suppose that the Legendrian rulings have $\infty$ slope on the boundary tori. From Lemma \ref{lem: NonRotTightContStr}, we know that there is a convex annulus $A$ with different configurations of the dividing curves for each integer $n\in \mathbb{Z}$. 
In Figure \ref{fig: gluingAnnuliToOD}, we have shown one such configuration corresponding to $n=0$ in the Lemma \ref{lem: NonRotTightContStr}. 
On the glued thickened torus $\mathbb{T}^2 \times [1,2]$, there is a convex annulus $B$ with two boundary parallel dividing curves on both faces, bounding a positive singularity as shown in the Figure \ref{fig: gluingAnnuliToOD}. 
In our contact round 1-surgery, annulus $A$ is glued to $B$, and we obtain a null-homotopic dividing curve on the convex torus. By Giroux's criterion (Theorem \ref{thm: Giroux Criterion}), a neighborhood of this torus is not tight. Hence, the resultant 3-manifold is overtwisted.

\subsubsection{Realizaton of Adachi's contact round surgery of index 2 on $\left(\s^3, \xi_{st}\right)$ as a contact round 2-surgery performed on a Legendrian knot $K\subset \left(\s^3, \xi_{st}\right)$} 
Adachi introduced a version of contact round 2-surgery in \cite{Jiro19}  on a contact 3-manifold $(M, \xi)$ as an attachment of the symplectic round handle to the concave end along a convex torus of the symplectization $(M\times [0,1], d(e^t\alpha))$.

In particular, Adachi's contact round surgery of index 2 can be described as follows. 
Suppose $T\subset (M, \xi)$ is the standard convex torus with two parallel dividing curves on it. 
Choose a simple closed curve $m \in H_1(T, \mathbb{Z})$ to be a meridian such that it intersects each dividing curve once. 
The chosen meridian $m$ in the surgery is called the {\it surgery meridian}. 
The dividing curve and the meridian $m$ give a coordinate system on the torus $T$. 

\begin{enumerate}
    \item Remove the interior $T\times (-\epsilon, \epsilon)$ of the invariant neighborhood of $T\subset(M, \xi)$. 
    \item Reglue the two standard tubular neighborhoods of Legendrian knots to $T\times \{\pm\epsilon\}$ $\subset \dd \{M\setminus (T\times (-\epsilon, \epsilon))\}$ so that dividing curves and the meridians agree with the same on $T\times \{\pm\epsilon\}$, respectively. 
\end{enumerate}

\begin{thm}\label{thm: JCR2SasCR2S}
Adachi's contact round 2-surgery on a convex torus $T\subset \mathbb{S}^3$ is the same as contact round 2-surgery on a Legendrian knot $K \subset \s^3$ with contact round surgery coefficient $n$ for some $n\in \mathbb{Z}$, where $\dd N_{\delta}(K)= T$ for  $\delta>0$.   Moreover, the integer $n$ depends only on the choice of the surgery meridian.   
\end{thm}

\begin{proof}
We know there is an ambiguity in the choice of meridian while gluing a solid torus in Adachi's contact round 2-surgery. 
By construction, the surgery meridian is an image of the meridian of the glued solid torus.
In our description of contact round 2-surgery, the round 2-surgery coefficient corresponds to the image of the meridian. 
Thus, the surgery meridian and the curves corresponding to the round surgery coefficient are the same.
In our definition of contact round 2-surgery, there is a natural choice of meridian $\mu$. 
We take $T\times \{0\}= \dd N(K)$ or $T\times [-\epsilon, \epsilon] = N_{\delta_2}(K) \setminus int(N_{\delta_1}(K))$.  
We take $\mu$ as the closed curve on $\dd N_{\delta_2} (K)$ that bounds a meridional disk in $N_{\delta_2}(K)$, and contact longitude $\lambda_c$ is given by a dividing curve. 
Suppose we have an identification $T\times \{0\}$ with $\mathbb{R}^2/\mathbb{Z}^2$ with $\lambda_c \mapsto (1,0)$ and $\mu \mapsto (0,1)$. 
Since surgery meridian $m$ intersects each dividing curve once, the coefficient of contact longitude is $1$. 
We may express the surgery meridian as $m = n \cdot \mu + \lambda_c$.
Therefore, the contact round 2-surgery coefficient is $n$.
\end{proof}

\section{A Correspondence Between the Contact Round Surgery Diagrams and Contact Dehn Surgery Diagrams}

In the following, we use the term round surgery and contact round surgery to emphasize the difference between topological round surgery and contact round surgery. 
In this section, we define a joint pair of contact round surgeries of indices 1 and 2 so that contact round surgeries on joint pairs yield a \emph{connected} contact 3-manifold.  
Suppose $L= L_{1} \cup L_{2}$ is a Legendrian link such that it is a joint pair of round surgeries of indices 1 and 2 with round 1-surgery coefficient $n_{j}$ on $L_{j}$, for $j=1, 2$, and round 2-surgery coefficient $m=\frac{p}{q}$ on $L_{2}$ with respect to the respective contact longitudes (See Figure \ref{fig: ContJP}).
We consider standard tubular neighborhoods $N_{\delta_2}(L_{j})$ of Legendrian link components $L_{j}$ such that $N_{\delta_2}(L_{1})\cap N_{\delta_2}(L_{2})=\emptyset$. By definition of a joint pair, we perform a round 1-surgery on $int(N_{\delta_1}(L_{1}))\cup int(N_{\delta_1}(L_{2}))$ and a round 2-surgery on $N_{\delta_2}(L_{2})\setminus int(N_{\delta_1}(L_{2}))$. 

\begin{defn}
    A joint pair $L_{1}\cup L_{2} \subset \left(\s^3, \xi_{st}\right)$ is said to be a {\it contact joint pair} if 
    \begin{enumerate}
        \item $L_{1}$ and $L_{2}$ are both Legendrian knots, 
        \item and the integers $n_{1}$ and $n_{2}$ denote the contact round 1-surgery coefficients on $L_{1}$ and $L_{2}$, respectively, and rational number $m$ denotes a contact round 2-surgery coefficient on $L_{2}$. 
    \end{enumerate}
    We perform a contact round 1-surgery on  $int(N_{\delta_1}(L_{1}))\cup int(N_{\delta_1}(L_{2}))$ and a contact round 2-surgery on $N_{\delta_2}(L_{2})\setminus int(N_{\delta_1}(L_{2}))$, where $0< \delta_1 < \delta_2$. 
\end{defn}

For contact round 1-surgery, we remove $int(N_{\delta_1}(L_{1}))\cup int(N_{\delta_1}(L_{2}))$ from $\left(\s^3, \xi_{st}\right)$ and glue $\mathbb{T}^2 \times [1,2]$ with some tight contact structure as discussed in Subsection \ref{subsec: CR1S}. 
For the contact round 2-surgery, we remove $N_{\delta_2}(L_{2})\setminus int(N_{\delta_1}(L_{2}))$ from the resulting 3-manifold. 
As a result, we obtain $N$ with some contact structure. 
The 3-manifold with boundary $N$ is diffeomorphic to $\overline{\s^3 \setminus int(N(L))}$ from Lemma~\ref{lem: RemovedTorusNbd}. 
Observe that $\dd N = \dd N_{\delta_2}(L_{2}) \sqcup \mathbb{T}^2 \times \{2\}$. 
We glue one copy of the solid torus along each boundary component, so two in total.
We denote the boundary torus by $T_j$ if it is glued to $\dd N_{\delta_j}(L_{2})$ for $j=1,2$. 
On $\dd T_j$, suppose $m_j$ denotes the meridian, and it maps to $s_j$. 
On $\dd N_{\delta_2}(L_{2})$,  $m_2 \mapsto p \cdot \mu_2 + q \cdot \lambda_{c_2}$. 
In round  1-surgery, the boundary torus $\dd N_{\delta_1}(L_{2})$ is identified with $\mathbb{T}^2 \times \{2\}$ and the dividing curve $x+ (-n_{2})\cdot y $ is glued to $\lambda_{c_2}$. 
On $\mathbb{T}^2\times \{2\}$, $s_1 = p\cdot \mu_2 + q \cdot \lambda_{c_2}= (p-n_{2})\cdot y +q \cdot x $. 
We glue $T_j$ to $\dd N_{\delta_j}(L_{2})$ by mapping the meridian $m_j$ to $s_j$ as per the procedure of contact round 2-surgery in Subsection \ref{subsec: CR2S}. 

\begin{figure}[ht]
    \centering
    \includegraphics[scale=0.4]{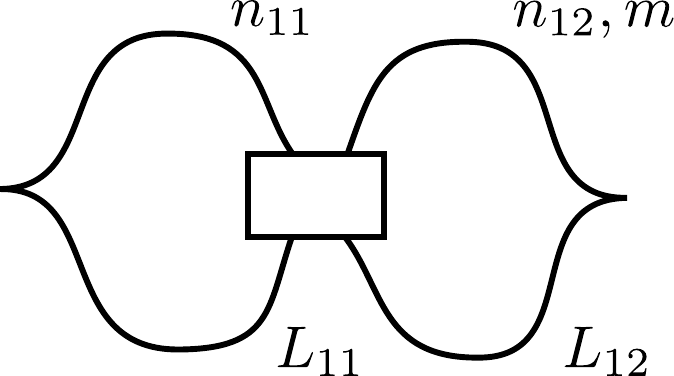}
    \caption{Schematic front diagram of a contact joint pair $L_{1} \cup L_{2}$ with contact round 1-surgery coefficients $n_{1}$ on $L_{1}$ and $n_{2}$ on $L_{2}$, and contact round 2-surgery coefficient $m$ on $L_{2}$. The middle box in the figure represents all the knotting of individual components and the linking of components.}
    \label{fig: ContJP}
\end{figure}

\iffalse
On $\mathbb{T}^2 \times \{2\}$, we have a coordinate system given by $x$ and $y$. 
The dividing curve $\lambda_{c_2}$ maps to $-n_{12}\cdot y + x$. We denote $\lambda_{c_2}^{new} :=-n_{12}\cdot y + x$ as the new contact longitude of $\mathbb{T}^2 \times \{2\}$. 
Thus, we can express the round 2-surgery gluing curve $s$ as follows. 
The gluing curve is denoted by $s = p \cdot \mu_2 + q \cdot \lambda_{c_2}$. 
Since $\lambda_{c_2}^{new}= -n_{12} \cdot y + x$, we get $s_1 = (p-n_{12}q)\cdot x+ q \cdot \lambda_{c_2}^{new}$. 
Further, on $\dd N_{\delta_2}(L_{i2})$, the coordinate curves remain $\mu_2$ and $\lambda_{c_2}$ which implies the surgery curve $s_2= p\cdot \mu_2 + q\cdot \lambda_{c_2}$.\fi

\iffalse
\begin{rem}
    In a contact Dehn surgery on a Legendrian link $L= L_1 \cup L_2 \cup \cdots\cup  L_n$, we remove the standard tubular neighborhoods and obtain $\s^3 \setminus int(N(L))$. 
    The restriction of the $\xi_{st}$ induces a tight contact structure on $\s^3 \setminus int(N(L))$ with some boundary conditions. 
\end{rem}
\fi

Suppose $L=L_1\cup \cdots \cup L_n$ is a Legendrian link.
We take standard tubular neighborhoods $N(L_i)$ of each $L_i$ and remove their interior from $\left(\s^3, \xi_{st}\right)$. 
As a result, we get $M= \s^3 \setminus \{int(N(L_1))\cup \cdots \cup int(N(L_n))\}$. 
$M$ is a 3-manifold with $n$-many toroidal boundary components. 
\iffalse
\begin{lem}\label{lem: Remove_N(L)}
    Each boundary torus $\dd N(L_i)$ is a convex torus with two parallel dividing curves such that the slope is equal to the reciprocal of the Thurston-Bennequin invariant of $L_i$,i.e., $1/tb(L_i)$. % with respect to the canonical coordinates on $\dd N(L_i)$. 
\end{lem}
\fi 

\begin{lem}\label{lem: Remove_N(L)}
    $\s^3 \setminus int (N(L))$ has a tight contact structure induced from the $\xi_{st}$ with boundary tori satisfying following conditions. 
    
    \begin{enumerate}
        \item Each boundary torus $\dd N(L_i)$ is a convex torus with two parallel dividing curves 
        \item and the boundary slopes are equal to the reciprocal of the Thurston-Bennequin invariant of $L_i$, i.e., $1/tb(L_i)$, with respect to the canonical coordinates on $\dd N(L_i)$.
    \end{enumerate}
\end{lem}

In particular, $\s^3 \setminus int (N(L))$ has a tight contact structure induced from $\xi_{st}$ with boundary tori satisfying the conditions mentioned in the above Lemma \ref{lem: Remove_N(L)}. 

\begin{proof}[Proof of the Lemma \ref{lem: Remove_N(L)}]
  Since we have removed the interior of the standard tubular neighborhoods $N(L_i)$ of each $L_i$ from $\left(\s^3, \xi_{st}\right)$, on $\dd N(L_i)$, the $\dd N_{\delta_1}(L_i)$ is a convex torus with two dividing curves isotopic to contact longitude $\lambda_{c_i}$. 
    Since $L_i \subset \s^3$, we can express $\lambda_{c_i}$ in terms to canonical longitude $\lambda_i$ as $\lambda_{c_i}= tb(L_i)\cdot \mu_i + \lambda_i$. 
    We identify boundary torus $\dd N(L_i)$ with $\mathbb{R}^2/\mathbb{Z}^2$ by mapping $\mu_i\mapsto (1,0)$ and $\lambda_i \mapsto (0,1)$. 
    We get slope $s(\dd N(L))=1/tb(L_i)$ under this identification. 
    Therefore, $\s^3 \setminus int(N(L))$ has tight contact structure induced from the $\xi_{st}$ with boundary slope $1/tb(L_i)$. 
\end{proof}

\begin{rem}
    The above boundary conditions are not sufficient to determine a tight contact structure on $\s^3 \setminus int(N(L))$ to be the restriction of $\xi_{st}$. 
    For example, consider $L$ to be a Hopf link. We get a thickened torus $\mathbb{T}^2 \times I$ as $\s^3 \setminus int(N(L))$. 
    By Ko Honda's classification of the tight contact structure on $\mathbb{T}^2 \times I$ in \cite{KH}, we know that there are infinitely many tight contact structures on $\mathbb{T}^2 \times I$ satisfying the same boundary conditions. 
\end{rem}

Let the Legendrian link $L_{1}\cup L_{2}$ be a joint pair with round 1-surgery coefficient $n \in \mathbb{Z}$ on $L_{1}$ and $L_{2}$ and $m\in \mathbb{Q}$ on $L_{2}$. 
From Lemma~\ref{lem: RemovedTorusNbd}, we know that the effect of round 1-surgery and a removal of the embedded thickened torus for round 2-surgery is $M = \s^3 \setminus \{int(N(L_{1}))\cup int (N(L_{2}))\}$. 
In the following lemma, we prove that by a specific choice of tight contact structure on glued $\mathbb{T}^2 \times I $ for round 1-surgery corresponds to $(M, \xi_{st}|_M)$. 

\begin{lem}\label{lem: ResOfStdTightCont}
    Suppose, in the above setting, the glued thickened torus $\mathbb{T}^2 \times [1,2]$ in contact round 1-surgery has a tight contact structure satisfying the following conditions:  
    \begin{enumerate}
        \item The boundary tori $\mathbb{T}^2 \times \{1\}$ and $\mathbb{T}^2 \times \{2\}$ are convex with two parallel dividing curves of slopes $n\in \mathbb{Z}$ on each of them,
        \item and $\mathbb{T}^2 \times I$ has a minimal twisting non-rotative tight contact structure corresponding to the preimage of $0$ under the holonomy map defined in Lemma \ref{lem: NonRotTightContStr}.  
    \end{enumerate}
    Then $M $ has a restriction of the standard tight contact structure from $\s^3$. 
\end{lem}
\begin{proof}
    From Lemma~\ref{lem: RemovedTorusNbd}, we know that $M \cong \s^3 \setminus int(N(L))$. 
    Thus, we only need to show that the contact structure $\xi$ on $M$ is contactomorphic to $\xi_{st}$.

    {\bf Claim 1.} The boundary tori $\dd M = \dd N(L_{1}) \cup \dd N(L_{2})$ satisfy the necessary boundary conditions.
    
    {\it Proof of claim 1.} It is sufficient to prove that the contact structure on $M$ satisfies conditions mentioned in Lemma \ref{lem: Remove_N(L)}.
    By definition of contact joint pair, the boundary torus $\dd N_{\delta_2}(L_{2})$ is a convex torus obtained as a boundary of the standard tubular neighborhood of $L_{2}$. 
    Thus, the slope $s(\dd N_{\delta_2}(L_{2}))= 1/tb(L_{2})$. 

    On $\dd N_{\delta_1}(L_{2})$, which is glued to $\mathbb{T}^2 \times \{2\}$, we have two sets of coordinates curves: one with contact longitude  $\lambda_{c_1}$ and meridian $\mu_1$ given by the standard neighborhood of $L_{2}$ and other is given by the natural coordinates of $\mathbb{T}^2 \times [1,2]$ with meridian $x$ and longitude $y$.
    Recall that, under the gluing of round 1-surgery $x \mapsto \mu_2$ and $y \mapsto n\cdot \mu_2 + \lambda_{c_2}$. Thus, the dividing curve is given by $\lambda_{c_2}^{new}= -n\cdot x +y$. 
    On $\dd N_{\delta_1}(L_{1})$, we have $x \mapsto \mu_1$ and $y \mapsto n\cdot \mu_1 + \lambda_{c_1}$. 
    It implies that $-n\cdot x + y \mapsto - n\cdot\mu_1 + n\cdot \mu_1 + \lambda_{c_1}= \lambda_{c_1}$. 
    Therefore, the slopes of each boundary torus are given by $1/tb(L_{j})$. 
    Hence, $\xi$ is a tight contact structure on $\s^3\setminus int(N(L))$ satisfying the necessary boundary conditions. 

{\bf Claim 2.} The glued thickened torus in $M$ is contactomorphic to the $I$-invariant neighborhood of $\dd N_{\delta_1}(L_{1})$. 

{\it Proof of claim 2.} %We know that $\mathbb{T}^2 \times \{2\}$ is a convex torus with two dividing curves of slope $1/tb(L_{11})$ with respect to the canonical coordiantes and $\s^3 \setminus (\mathbb{T}^2 \times I)$ has the restriction of $\xi_{st}$. 
By Lemma \ref{lem: NonRotTightContStr}, $\mathbb{T}^2 \times I$ with minimal twisting non-rotative tight contact structure corresponding to the preimage of $0$ under the holonomy map is an $I$-invariant neighborhood of  $\mathbb{T}^2 \times \{3/2\}$. 
We glue $\mathbb{T}^2\times \{1\}$ to $\dd N_{\delta_1}(L_{1})$ such that the dividing curves of $\mathbb{T}^2\times \{1\}$ maps to the dividing curves of $\dd N_{\delta_1}(L_{1})$. 
Since $\mathbb{T}^2 \times \{1\}$ is isotopic to $\mathbb{T}^2 \times \{\frac{3}{2}\}$, the glued $\mathbb{T}^2 \times I$ can be realized as an $I$-invariant neighborhood of $\dd N_{\delta_1}(L_{1})$. 
Thus, the proof of Claim 2 is complete. 

The contact structure on $M$ is the restriction of $\xi_{st}$ in the complement of the glued thickened torus. 
By Claim 2, the glued thickened torus does not change that contact structure on $M$.
Thus, the contact structure on $(M, \xi) \subset \left(\s^3, \xi_{st}\right)$ is the restriction of $\xi_{st}$. 
%After performing contact round 1-surgery and gluing an $I$-invariant neighborhood of $\mathbb{T}^2 \times \{\frac{3}{2}\}$, we get $M$ with 
%$\dd N_{\delta_1}(L_{11})$ are boundaries of this neighborhood. 
%Since the glued thickened torus is contactomorphic to $\dd N_{\delta_1}(L_{11})\times I$, the contact structure on $M$ is contactomorphic to $\xi_{st}|_M$. 
\end{proof}

\begin{lem}\label{lem: RealOfContJointPairAsContDehnSur}
   Let $L_{1}\cup L_{2}$ be a contact joint pair with round 1-surgery coefficient $n$ on both components and round 2-surgery $m$ on $L_{2}$, where $m\in \{\pm1\}$. 
    Assume that after the round 1-surgery and removing thickened torus $N_{\delta_2}(L_{2})\setminus N_{\delta_1}(L_{2})$ the contact structure on $\s^3 \setminus N(L)$ is $\xi_{st}|_{\s^3 \setminus N(L)}$. 
   Then the round surgery along the joint pair $L_{1}\cup L_{2}$ is equivalent to contact $m$-surgeries along the link.
\end{lem}

\begin{proof}
From Claim 1 of Lemma~\ref{lem: ResOfStdTightCont}, we obtain $\s^3 \setminus int (N(L))$ with a tight contact structure $\xi_{st}|_{\s^3 \setminus int(N(L))}$ with boundary torus having dividing curves parallel to $\lambda_{c_j}$ on $\dd N(L_{j})$. 
In the round 2-surgery, we glue solid torus along each boundary component $\dd N(L_{j})$ by mapping $\{p\} \times \dd D^2$ to the curve $m\cdot \mu_j + \lambda_{c_j} $. 

For each $L_{j}$, we can realize this gluing as a contact $ m$-surgery along $L_j$. 
Thus, round surgery on a given contact joint pair can be realized as a pair of contact Dehn surgeries with coefficient $m$.
\end{proof}

\begin{defn}
    We call a contact joint pair $L_{1}\cup L_{2}$ {\it nice} if it satisfies the following conditions. 
   \begin{enumerate}
       \item The round 1-surgery coefficients on both components are the same integer $n$ and the round 2-surgery coefficient $m$ on $L_{2}$, where $m\in \{\pm1\}$.
       \item After performing the round 1-surgery and removing thickened torus $N_{\delta_2}(L_{2})\setminus N_{\delta_1}(L_{2})$ the contact structure on $\s^3 \setminus int(N(L))$ is $\xi_{st}|_{\s^3 \setminus int(N(L))}$.
   \end{enumerate}
\end{defn}

With the above definition, we now state the following correspondence between contact round surgeries and contact Dehn surgeries.

\begin{thm}[Contact Bridge Theorem]\label{thm: ContBridgeThm}
The following two statements establish a bridge between a set of contact $(\pm1)$-surgery diagrams and contact round surgery diagrams of nice contact joint pairs. 
\begin{enumerate}
    \item [(1)] For a contact $(\pm1)$-surgery diagram on a Legendrian link $L_1 \cup \cdots \cup L_n \subset \s^3$ of a contact 3-manifold $(M, \xi)$, there is a round surgery diagram consisting of contact joint pairs $L= \bigcup_{i=1}^{n'} (L_{i1}\cup L_{i2})$ with round 1-surgery coefficients $k\in \mathbb{Z} $ on both components $L_{1j}$ and round 2-surgery coefficient $m \in \{\pm1\}$ on $L_{i2}$. 
    \item [(2)]Given a contact round surgery diagram of nice contact joint pairs $L= \bigcup_{i=1}^{n} (L_{1i}\cup L_{i2})$ satisfying the following conditions: 
    \begin{enumerate}
        \item The round 1-surgery coefficient is $k \in  \mathbb{Z}$ on both components, 
        \item and round 2-surgery coefficient is $m\in \{\pm1\}$ on $L_{i2}$. 
        %\item after removing thickened torus for round 2-surgery on a 2-manifold obtained by round 1-surgery, we get $\s^3\setminus int(N(L))$. We assume tight contact structure on $\s^3 \setminus int(N(L))$ is $\xi_{st}|_{\s^3 \setminus int(N(L))}$. 
    \end{enumerate}
    The Legendrian link $L= \bigcup_{i=1}^{n} (L_{1i}\cup L_{i2})$ determines a contact $(\pm1)$-surgery diagram such that for each $i,j$; $L_{ij}$ has contact surgery coefficient $m$. 
\end{enumerate}
    
\end{thm}

\begin{proof}[Proof of (1)]We have a Legendrian link $$L= \left(\bigcup_{i=1}^{n_1}L_{i}^{+1}\right)\cup \left(\bigcup_{j=1}^{n_2}L_{j}^{-1}\right),$$ where we perform contact $(+1)$-surgery on $L_{i}^{+1}$ (and contact $(-1)$-surgery on $L_{j}^{-1}$). 
    We consider the following cases. 
    \begin{enumerate}
        \item The numbers $n_1$ and $n_2$ are even.
        In this case, we pair the $(+1)$-surgery components (and $(-1)$-surgery components) together. 
        We obtain $$L= \bigcup_{i=1}^{\frac{n_1}{2}} (L_{2i-1}^{+1}\cup L_{2i}^{+1}) \bigcup_{j=1}^{\frac{n_2}{2}}(L_{2j-1}^{-1}\cup L_{2j}^{-1}).$$
        We treat the pair $L_{2i-1}^{+1}\cup L_{2i}^{+1}$ (or $L_{2j-1}^{-1}\cup L_{2j}^{-1}$)  as a nice contact joint pair with some contact round 1-surgery coefficient $k\in \mathbb{Z}$ and contact round 2-surgery coefficient $+1$ on $L_{2i}^{+1}$ (or $-1$ on $L_{2j}^{-1}$). 
        We get back the given contact $\pm1$ surgery pair after applying Lemma \ref{lem: RealOfContJointPairAsContDehnSur} on these nice contact joint pairs.

        \item The number $n_1$ is odd but $n_2$ is even. 
        In this case, we add an appropriate contact surgery diagram via the contact Kirby move of type 1 as discussed in Subsection \ref{subsec: Pre-CFKM}. 
        We consider an unknot $U$ with $tb(U)=-2m$ and contact surgery coefficient $2m+1$ on it as shown in Figure \ref{fig:NewDiagram}, for some positive integer $m$.

       We take its contact $(\pm1)$-surgery presentation corresponding to the standard contact 3-sphere. 
        This presentation has one contact $(+1)$-surgery unknot $U$ and $2m$-many contact $(-1)$-surgery unknots. 
        After applying a contact Kirby move of type 1 to the link $L$, we can pair components with $(+1)$-surgery coefficients together by taking two components at a time. 
        We repeat the same process with components carrying $(-1)$-surgery coefficients. Then, we apply the first case.

        \item The number $n_1$ is even but $n_2$ is odd. 
        In this case, we apply contact Kirby moves of type 1 twice by adding appropriate integral contact cosmetic surgery diagrams to the given link $L$. 
        The first diagram consists of an unknot $U_1$ with $tb(U_1)=-2m_1-1$ and contact surgery coefficient $2m_1+2$, and the second diagram is an unknot $U_2$ with $tb(U_2)=-2m_2$ and contact surgery coefficient $2m_2+1$, for some positive integers $m_1$ and $m_2$.
        We add their contact $(\pm1)$-surgery presentation corresponding to $\left(\s^3, \xi_{st}\right)$ to $L$. 
        We obtain $(n_1+ 2)$-many contact $(+1)$-surgery components and  $(n_2+ 2m_1+1 + 2m_2)$-many contact $(-1)$-surgery components.
        We pair components with $(+1)$-surgery coefficients together by taking two components at a time and repeat the same for the components with $(-1)$-surgery coefficients. 
        Then, we apply the first case.

        \item The numbers $n_1$ and $n_2$ are odd. 
        We apply a contact Kirby move of type 1 by adding a contact $(\pm1)$-surgery presentation of unknot $U_1$ with $tb(U_1)=-2m-1$ and contact surgery coefficient $2m+2$ corresponding to $\left(\s^3, \xi_{st}\right)$ for some positive integer $m$.
        Since this presentation has one contact $(+1)$-surgery component and $(2m+1)$-many contact $(-1)$-surgery components, we get an even number of components for both surgery coefficients. 
        We pair components with $(+1)$-surgery coefficients together by taking two components at a time and repeat the same for the components with $(-1)$-surgery coefficients.
        Then, we apply the first case.
    \end{enumerate}
   
\noindent {\it Proof of (2).} We have a Legendrian link $L = \bigcup_{i=1}^{n}(L_{i1} \cup L_{i2})$ of nice contact joint pairs. 
    Then, by Lemma \ref{lem: RealOfContJointPairAsContDehnSur}, each pair can be realized as a pair of contact $m$-surgeries on its components $L_{i1}\cup L_{i2}$ and hence, giving a contact $(\pm1)$-surgery diagram.
\end{proof}

\begin{defn}
    A Legendrian link $L_1 \cup \cdots \cup L_{n}$ in $\left(\s^3, \xi_{st}\right)$ is called a \textit{contact round surgery presentation} of a contact 3-manifold $(M, \xi)$ if 
    \begin{enumerate}
       \item A contact round 1-surgery is performed on a sublink of two components with some integer contact round 1-surgery coefficients on each component and a specified tight contact structure on the glued thickened torus.
        \item A contact round 2-surgery is performed on a link component (which is either in a contact joint pair with some round 1-surgery link or not) with a rational coefficient and a specified choice of tight contact structure on the gluing solid tori. 
        \item $(M, \xi)$ is obtained by performing contact round surgeries of index 1 and 2.
    \end{enumerate}
\end{defn}

Observe that a round surgery presentation of a connected contact 3-manifold has contact round 2-surgery knots in the joint pair with some round 1-surgery links; otherwise, they may yield disconnected components. 
Thus, a round surgery presentation of a connected contact 3-manifold has an even number of components. 
In the following corollary, we get a contact round surgery presentation of any closed, connected, oriented contact 3-manifold by nice contact joint pairs. 

\begin{cor}[Ding--Geiges Theorem for contact round surgery diagrams.]\label{cor: DGThmForCRSD}
Any closed, oriented, connected contact 3-manifold has a contact round surgery presentation in $\left(\s^3, \xi_{st}\right)$.
\end{cor}

\begin{proof}
    Recall that any closed, oriented contact 3-manifold can be obtained by a contact $(\pm1)$-surgery on some Legendrian link $L$ by Ding--Gieges' theorem. 
    We use the contact bridge theorem to get corresponding contact round surgery diagrams of nice contact joint pairs. 
\end{proof}

\begin{cor}
    Suppose $L$ denotes a contact round surgery diagram consisting of nice contact joint pairs $\bigcup_{i=1}^{n}(L_{i1}\cup L_{i2})$ with contact round 1-surgery coefficient $n\in \mathbb{Z}$ on $L_{ij}$ and contact round 2-surgery coefficient $-1$ on $L_{i2}$. Then, $L$ describes a symplectically fillable contact 3-manifold.  
\end{cor}

\begin{proof}
    We apply Theorem \ref{thm: ContBridgeThm} on the given Legendrian link $L$. 
    We get a contact surgery diagram consisting of contact $(-1)$-surgery components. 
    Therefore, we obtain a symplectically fillable 3-manifold after performing contact round surgery on the given diagram. 

    In the following, we discuss the construction of a particular symplectic filling of the resulting 3-manifold explicitly.
 First, we give the topological description of the handle attachment corresponding to the surgery diagram given in the hypothesis.
Let $X^4 := \s^3 \times [0,1]$. Then $\dd X^4= \s^3\times \{0,1\}$. 
We consider a topological joint pair $K_1 \cup K_2$ in $\s^3\times \{1\}$ with round $1$-surgery coefficients $n_{1}$ and $n_{2}$, and a round $2$-surgery on $K_{2}$ with an integer round $2$-surgery coefficient $m$.

%In order to understand the effect of round surgery along a joint pair in terms of handle attachments, we recall the (topological) notion of a joint pair along a pair of knots $K_{1} \cup K_{2}$.

In the $4$-dimensional setup, we attach a round $1$-handle $R^{1}$, which is a diffeomorphic copy of $\s^1\times \D^1\times \D^2$, to $X^4$ along the attaching region $\s^{1} \times \s^{0} \times \D^{2}$, as shown in Figure~\ref{fig: R1,2H}. 
In Figure~\ref{fig: R1,2H}, we use bricks to model round handles of indices $1$ and $2$. 
Then, we attach a round $2$-handle $R^{2}$, which is a diffeomorphic copy of $\s^{1} \times \D^{2} \times \D^{1}$, along its attaching region $\s^{1} \times \partial \D^{2} \times \D^{1}$.

\begin{figure}[ht]
    \centering
    \includegraphics[width=0.5\linewidth]{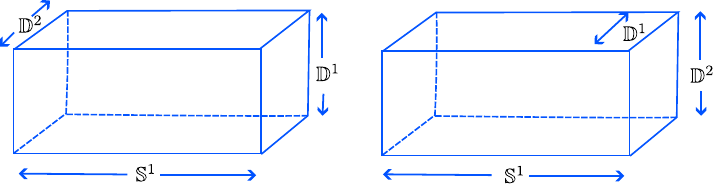}
    \caption{The left and right bricks represent schematic diagrams of a round $1$-handle and a round $2$-handle, respectively. The left and right faces of each brick are identified. The top and bottom faces of each brick indicate the attaching regions of the handles, while the front and back faces represent their belt regions.}
    \label{fig: R1,2H}
\end{figure}

We recall that, in a joint pair, we perform the round $2$-surgery on a thickened neighborhood of a torus parallel to $\partial N(K_{2})$, which is equivalent to attaching $R^{2}$ to the belt region of the round $1$-handle. 
Furthermore, we attach it in a specific way to mimic the effect of the joint pair on the boundary. For this, we note the following.

While performing round $2$-surgery as part of a joint pair, the meridional disk of the gluing solid torus maps to a curve isotopic to the canonical longitude on the boundary $\partial N(K_{i})$ for each $i=1,2$.\\
This implies that we attach $\s^{1} \times \D^{2} \times \D^{1}$ to the belt region $\s^{1} \times \D^{1} \times \s^{1} \subset R^{1}$ by a gluing map 
\[
\phi: \s^{1} \times \partial \D^{2} \times \D^{1} \longrightarrow \s^{1} \times \D^{1} \times \s^{1}
\]
defined by
\[
\phi\left(\{p\} \times \partial \D^{2} \times \{t\}\right) = \s^{1} \times \{t\} \times \{p\},
\]
for $p \in \s^{1}$ and $t \in \D^{1}$. 

For the attachment, we refer to Figure~\ref{fig:round2handle}.

\begin{figure}[ht]
\begin{center}
\includegraphics[width=0.5\linewidth]{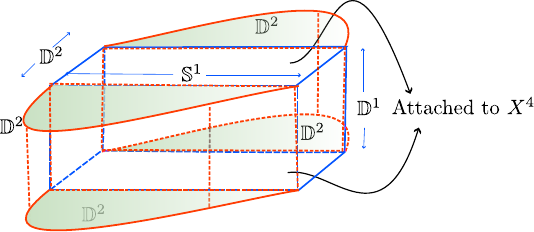}
\caption{Round $2$-handle (shown with a red boundary) attached along the belt region of a round $1$-handle (shown with a blue boundary). 
The top and bottom faces of the round $1$-handle are attached to $X^{4}$.}
\label{fig:round2handle}
\end{center}
\end{figure}

After the attachment of $R^{2}$, we get
\[
\begin{aligned}
X^{\prime} &= (X \cup R^{1}) \cup R^{2}, \quad \text{and} \\
\partial X^{\prime} &= \left[\partial X \setminus \left(\s^{1} \times \D^{2} \times \s^{0}\right)\right] 
\cup \left(\D^{2} \times \s^{1} \times \s^{0}\right).
\end{aligned}
\]

Hence,
\[
\begin{aligned}
X^{\prime} 
&= \left\{ X \bigcup_{\s^{1} \times \s^{0} \times \D^{2}} \left(\s^{1} \times \D^{1} \times \D^{2}\right) \right\} 
\bigcup_{\phi} \left(\s^{1} \times \D^{2} \times \D^{1}\right) \\
&= X \bigcup_{\s^{1} \times \s^{0} \times \D^{2}} 
\left\{ \left(\s^{1} \times \D^{1} \times \D^{2}\right) \bigcup_{\phi} \left(\s^{1} \times \D^{2} \times \D^{1}\right) \right\} \\
&= X \bigcup_{\s^{1} \times \s^{0} \times \D^{2}} 
\left\{ \left(\s^{1} \times \D^{1} \times \D^{2}\right) \bigcup_{\phi} \left(\D^{2} \times \D^{1} \times \s^{1}\right) \right\}.
\end{aligned}
\]

By the gluing map $\phi$ defined above, we obtain
\[
\left\{ \s^{1} \times \D^{1} \times \D^{2} \right\} 
\bigcup_{\phi} 
\left\{ \D^{2} \times \D^{1} \times \s^{1} \right\} 
= \D^{2} \times \D^{2} \times \s^{0}.
\]
Thus,
\[
X^{\prime} = X \bigcup_{\s^{1} \times \s^{0} \times \D^{2}} 
\left( \D^{2} \times \D^{2} \times \s^{0} \right),
\]
that is, it is equivalent to attaching two $2$-handles to $X^{4}$, as shown in Figure~\ref{fig: RealizeAsTwo2-handles}.

\begin{figure}[ht]
    \centering
    \includegraphics[width=0.5\linewidth]{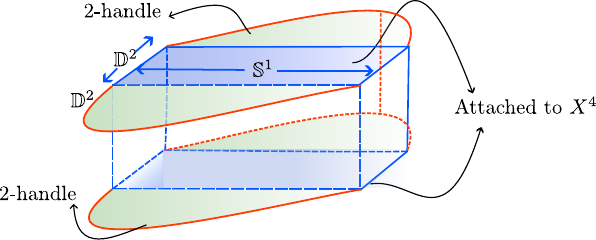}
    \caption{Each of the top and bottom faces of the diagram can be viewed as a $2$-handle attached along the blue-colored subface. 
The green-colored subfaces represent portions of the boundary of the resulting $4$-manifold.}
    \label{fig: RealizeAsTwo2-handles}
\end{figure}

Now consider $X^{4}$ with a symplectic form given by symplectization of the standard contact 3-sphere. 
By Part 2 of the Contact Bridge Theorem (Theorem~\ref{thm: ContBridgeThm}), the contact joint pair with contact round $1$-surgery coefficient $n$ on the Legedrian 
realization of $K_1$ and $K_2$, and contact round $2$-surgery coefficient $(-1)$ on the Legendrian realization of $K_2$ corresponds to contact $(-1)$-surgeries along each component $K_1$ and $K_2$. This is equivalent to performing Weinstein (symplectic) 2-handle attachment to the symplectic manifold $X^{4}$ to obtain $X'$.
Consequently, the symplectic structure extends naturally to the resulting $4$-manifold $X'$. Hence, the boundary manifold $\dd X'$ together with the induced contact structure is symplectically fillable.
\end{proof}

\bibliographystyle{plain}
\bibliography{Reference}

\end{document}